\documentclass[11pt]{article}

\usepackage{geometry}
\usepackage[applemac]{inputenc}
\usepackage[T1]{fontenc}
\usepackage{amsmath,amsfonts, amssymb,amsthm,url}
\usepackage[english]{babel}
\usepackage[algoruled,vlined,english,linesnumbered]{algorithm2e}
\usepackage[all]{xy}

\newcommand{\N}{\mathbb{N}}
\newcommand{\Z}{\mathbb{Z}}

\newcommand{\F}{\mathbb{F}}

\newcommand{\Hom}{\text{\rm Hom}}
\newcommand{\End}{\text{\rm End}}

\newcommand{\n}{\mathcal{N}}
\newcommand{\MM}{\text{\rm MM}}
\newcommand{\SM}{\text{\rm SM}}
\newtheorem*{theo}{Theorem}
{\theoremstyle{plain}
\newtheorem{thm}{Theorem}[subsection]}
\newtheorem{lem}[thm]{Lemma}
\newtheorem{prop}[thm]{Proposition}
\newtheorem{cor}[thm]{Corollary}
{\theoremstyle{definition}
\newtheorem{ex}[thm]{Example}
\newtheorem{rmq}[thm]{Remark}
\newtheorem{defi}[thm]{Definition}}

\begin{document}
\title{Some algorithms for skew polynomials over finite fields}
\author{Xavier Caruso, J{\'e}r{\'e}my Le Borgne}
\date{}
\maketitle
\begin{abstract}
In this paper, we study the arithmetics of skew polynomial rings over finite fields, mostly from an algorithmic point of view. We give various algorithms for fast multiplication, division and extended Euclidean division. We give a precise description of quotients of skew polynomial rings by a left principal ideal, using results relating skew polynomial rings to Azumaya algebras. We use this description to give a new factorization algorithm for skew polynomials, and to give other algorithms related to factorizations of skew polynomials, like counting the number of factorizations as a product of irreducibles.
\end{abstract}
\setcounter{tocdepth}{2}
\tableofcontents
\begin{center}
\line(1,0){300}
\end{center}
%\nomenclature[01]{$p$}{A prime number}
%\nomenclature[02]{$q$}{A power of the prime $p$}
%\nomenclature[03]{$r$}{An integer greater than or equal to $1$}
%\nomenclature[04]{$k$}{A finite field with cardinal $q^r$}
%\nomenclature[05]{$\sigma$}{An endomorphism of $k$ of order $r$}
%\nomenclature[06]{$k^\sigma$}{The subfield of $k$ fixed by $\sigma$}
%\nomenclature[07]{$k[X,\sigma]$}{The ring of skew polynomials with coefficients in $k$}
%\nomenclature[08]{$k^\sigma[X^r]$}{The centre of $k[X, \sigma]$}
%\nomenclature[09]{$P$}{An element of $k[X,\sigma]$}
%\nomenclature[10]{$N$}{An element of $k^\sigma[X^r]$}
%\nomenclature[11]{$d$}{The degree of the skew polynomial $P$}
%\nomenclature[12]{$\delta$}{The degree of the polynomial $N$}
%\nomenclature[13]{$\n$}{The reduced norm map on $k[X,\sigma]$}
%\nomenclature[14]{$D_P$}{The $\varphi$-module corresponding to $P$}
%\printnomenclature[5cm]
\section*{Introduction}
The aim of this paper is to present several algorithms to deal efficiently with rings of skew polynomials over finite fields. These noncommutative rings have been widely studied, including from an algorithmic point of view, since they were first introduced by Ore in 1933. The main applications for the study of skew polynomials over finite fields are for error-correcting codes.
The first significant results in terms of effective arithmetics in these rings, including an algorithm to factor a skew polynomial as a product of irreducibles, appear in Giesbrecht's paper \cite{gie2}. In this paper, we give a factorization algorithm whose complexity improves on Giesbrecht's. We also describe various fast-multiplication algorithms for skew polynomials, and some additional algorithms such as a factorization-counting algorithm, or an algorithm generating the uniform distribution on the factorizations of a given skew polynomial.\\
The first part of the article is mostly theoretical. Let $k$ be a finite field of characteristic $p$, and let $\sigma$ be an automorphism of $k$. We denote by $k^\sigma$ the subfield of $k$ fixed by $\sigma$, and by $r$ the order of $\sigma$ on $k$. The ring $k[X,\sigma]$ of skew polynomials with coefficients in $k$ is a noncommutative ring, on which multiplication is determined by $X \cdot a = \sigma(a) \cdot X$ for all $a \in k$. 
The first Theorem we shall prove is the following:
\begin{theo}[\emph{cf} Theorem \ref{azumaya}]
The ring $k[X,\sigma][1/X]$ is an Azumaya algebra over its centre $k^\sigma[X^r][1/X^r]$.
\end{theo}

This Theorem has many important consequences for our purpose. The first 
one is the existence of a \emph{reduced norm} map $k[X,\sigma] \to 
k[X^r]$, which turns out to have very nice properties related to 
factorizations. For instance, we shall explain how it can be used to 
establish a close link between factorizations of a skew polynomial and 
basic linear algebra over finite extensions of $k^\sigma$. 
As an illustration, we will show how to use this theory to derive a 
formula giving the number of factorizations of any skew polynomial.

The second part of the paper deals with algorithmic aspects of skew 
polynomials. We start by giving various fast-multiplication algorithms 
and, as usual, we derive from them efficient algorithms to compute 
Euclidean division and gcd. Then, we reach the core algorithm of this 
paper: the factorization algorithm. Making an intensive use of the 
theory developped in the first part, we obtain a very efficient 
algorithm to factor a skew polynomial as a product of irreducible skew 
polynomials, \texttt{SkewFactorization}.
\begin{theo}[\emph{cf} Theorem \ref{complexity}]
The algorithm \texttt{SkewFactorization} factors a skew polynomial of 
degree $d$ in $k[X,\sigma]$ with complexity
$$\tilde{O}(dr^3 \log q + d \log^2 q + d^{1+\varepsilon} (\log q)^{1 + o(1)} + F(d,k^\sigma))$$
bit operations, for all $\varepsilon > 0$. 
Here, $F(d,K)$ denotes the complexity of the factorization of a 
(commutative) polynomial of degree $d$ over the finite field $K$.
\end{theo}

In \cite{ku}, Kedlaya and Umans described a factorization algorithm
of polynomials over finite fiels whose complexity is:
$$F(d,K) = (d^{3/2 + o(1)} + d^{1 + o(1)}\log q) \cdot (\log q)^{1+o(1)}$$ 
bit operations, where $q$ is the cardinality of $K$. Assuming this value 
for $F(d,K)$, we see that the terms $d \log^2 q$ and $d^{1 + \varepsilon} 
(\log q)^{1 + o(1)}$ are negligible compared to $F(d,K)$. If 
furthermore $r^3 \ll d$, also is the term $dr^3 \log q$. With this extra 
assumption, the complexity of our algorithm is then comparable to the 
complexity of the factorization of a \emph{commutative} polynomial of 
the same degree.

The complexity of our algorithm should be compared to the complexity of 
Giesbrecht's algorithm, which is:
$$\tilde O(d^4 r^2 \log q + d^3 r^3 \log q + d \cdot \text{MM}(dr) \log q + d^2 r \cdot \log^2 q)$$
bit operations\footnote{In Giesbrecht's paper, the complexity is given
in number of operations in $k^\sigma$. Since any operation in $k^\sigma$
requires $\tilde O(\log q)$ bit operations (using fast algorithms), the
complexity we have given is just from Giesbrecht's one by multiplying
by $\tilde O(\log q)$.}
where $\text{MM}(n)$ is the complexity of the multiplication of two
$n \times n$ matrices.

The strategy of our algorithm is roughly comparable to the one of 
Giesbrecht's: in order to factor $P$, we find a multiple $N$ of $P$ 
lying in the centre of $k[X,\sigma]$, we factor $N$ in the centre (which 
is a commutative polynomial ring) and we recover a factorization of $P$ 
from the factorization of $N$ we have just computed.
The two main improvments are the following. Firstable, we obtain better
algorithms to achieve basic operations (like multiplication, Euclidean
division and gcd's). Using them to factor a polynomial improves
significantly the complexity. The second improvement (which is the most 
important) consists in taking a large benefit of the closed study of all 
involved objects we have done in the first part. For instance, in order 
to obtain the central multiple $N$, we just compute the reduced norm, 
for which efficient algorithms exist. In the same way, our theoretical 
results imply that, for some particular $P$, the quotient $k[X,\sigma] / 
k[X,\sigma] P$ is endowed with a rich structure and we use it to replace 
computations with large matrices over $k^\sigma$ by computations with 
matrices of size at most $r$ defined over a bigger field. Since usual
arithmetics in field extensions is more efficient than computations
with matrices (quasilinear \textsc{vs} subcubic), we gain a lot.

Eventually, we give an algorithm to compute the number of factorizations 
of a skew polynomial and we describe an algorithm to generate the 
uniform distribution on the factorizations of a skew polynomial.

All the algorithms described here have been implemented in \textsc{sage}, and some of them in \textsc{magma}. We discuss briefly about the implementation.

\medskip

This work was supported by the \emph{Agence Nationale de la Recherche},
CETHop project, number ANR-09-JCJC-0048-01.

\section{The ring $k[X,\sigma]$}
\subsection{Some facts about $k[X,\sigma]$}
Let $k$ be a finite field of characteristic $p$ and let $\sigma$ be an automorphism of $k$. We denote by $k^\sigma$ the subfield of $k$ fixed by $\sigma$. Let $r$ be the order of $\sigma$: $r$ is also the degree of the extension $k/k^\sigma$. We denote by $k[X,\sigma]$ the ring of skew polynomials with coefficients in $k$. The underlying group is just $k[X]$, and the multiplication is determined by the rule
$$ \forall a \in k, ~ Xa = \sigma(a)X.$$
We recall some notions from \cite{jac}, Chapter 1 (mainly \S 1.1 and 1.2). The centre of $k[X,\sigma]$ is $k^\sigma[X^r]$. The ring $k[X,\sigma]$ is endowed with left- and right-euclidean division algorithms. Hence, there are also notions of right- and left-greatest common divisor, and left- and right-lowest common multiple (denoted respectively by rgcd, lgcd, llcm, rlcm). Of course, every element of $k[X,\sigma]$ can be written as a product of irreducible elements of $k[X,\sigma]$. However, a such factorization is not unique in general. The first result that describes how two factorizations of a skew polynomial as a product of irreducibles are related is due to Ore. Before stating it, let us give a definition:
\begin{defi}
Let $P,Q \in k[X,\sigma]$ be two skew polynomials. Then $P$ and $Q$ are \emph{similar} if there exist $U,V \in k[X,\sigma]$ such that $\text{rgcd}(P,V) = 1$, $\text{llcm}(Q,U) =1$ and $UP = QV$.
\end{defi}
Even though it may not be clear at first glance, this is an equivalence relation. Remark that in the case $\sigma = \text{id}$, this just means that $P$ and $Q$ are equal up to multiplication by an element of $k^\times$.
We then have the following theorem:
\begin{theo}[Ore, \cite{ore}]
Let $P_1, \ldots, P_n$ and $Q_1, \ldots, Q_m$ be irreducible skew polynomials. If $P_1\cdots P_n = Q_1\cdots Q_m$, then $m = n$ and there exists a permutation $\tau$ of $\{1, \ldots, n\}$ such that for all $1 \leq i \leq n$, $P_i$ is similar to $Q_{\tau(i)}$.
\end{theo}
However, the converse of this theorem is false. In general, if the $P_i$ and $Q_i$ are pairwise similar, $\prod  P_i$ and $\prod Q_i$ are not even similar.\\
An interesting point of view on skew polynomials is that of $\varphi$\emph{-modules} that we shall elaborate on later. For now, it is enough to say that a $\varphi$-module over $k$ is a $k[X,\sigma]$-module of finite type. If $P \in k[X,\sigma]$ is nonzero, a typical example of a $\varphi$-module over $k$ is $k[X,\sigma]/k[X,\sigma]P$, which is ``the $\varphi$-module associated to $P$''. Then, two skew polynomials are similar if and only if the associated $\varphi$-modules are isomorphic, and Ore's theorem is just a restatement of the Jordan-H{\"older} Theorem in the category of $\varphi$-modules.
\subsection{The ring $k[X,\sigma][1/X]$ is an Azumaya algebra}
The aim of this section is to prove the Theorem \ref{azumaya} and to give several consequences. Let us now recall the statement of the Theorem:
\begin{thm}\label{azumaya}
The ring $k[X,\sigma][1/X]$ is an Azumaya algebra over $k^\sigma[X^r][1/X^r]$.
\end{thm}
\begin{proof}
Let us denote by $\mathcal{R}$ the ring $k[X,\sigma][1/X]$ and by $\mathcal{C}$ its centre $k^\sigma[X^r][1/X^ r]$. By definition, it is enough to show that for every prime ideal $\mathfrak{P}$ of $C$, $\mathcal{R/\mathfrak{P}} \otimes_\mathcal{C} \text{Frac}(C/\mathfrak{P})$ is a central simple algebra over $\text{Frac}(C/\mathfrak{P})$. The case $\mathfrak{P} = (0)$ is exactly \cite{jac}, Theorem 1.4.6. The other prime ideals of $\mathcal{C}$ are of the form $(N)$ with $N \in k^\sigma[X^r]$ monic irreducible and different from $X^r$. Fix such an irreducible polynomial $N$. Denote by $E$ the field of fractions of $\mathcal{C}/(N)$. Let us first show that $\mathcal{R}_N = \mathcal{R}\otimes_\mathcal{C}E$ is simple. Let $I \subset \mathcal{R}_N$ be a two-sided ideal. Assume that $I \neq (0)$, and let $x \in \mathcal{R}_N$ be a nonzero element of $I$. First remark that every element of $\mathcal{R}_N$ can be written as $P \otimes 1$. Indeed, if $t$ is the class of $X^r$ is $E = \mathcal{C}/(N)$, then $1 \otimes t = X^r \otimes 1$. Therefore, we can write $x = P \otimes 1$ with $P \in k[X,\sigma]/(N)$. Now assume that $x$ and $P$ are chosen such that the number of nonzero coefficients of $P$ is minimal (with $x \in I \setminus \{0\}$). We can assume that $P$ is monic of degree $d$. We have $P - XPX^{-1} \in I$, and this polynomial has less nonzero coefficients than $P$, so that it is zero. Similarly, if $a \in k^\times$, $P - \sigma^d(a)^{-1}Pa = 0$. This shows that $x$ is central. Since the centre of $\mathcal{R}_N$ is a commutative finite integral $E$-algebra, it is a field, so $x$ is invertible and $I = \mathcal{R}_N$.\\
It remains to prove that this centre is exactly $E$. We just need to solve the equations $X\sum_{i = 0}^{\deg(N) -1} a_i X^i = \sum_{i = 0}^{\deg(N) -1} a_i X^{i+1}$ and $\alpha\sum_{i = 0}^{\deg(N) -1} a_i X^i = \sum_{i = 0}^{\deg(N) -1} a_i X^{i}\alpha$ for $\alpha$ a generator of $k/k^\sigma$. It is easy to see that the solutions are exactly (the reduction modulo $N$ of) elements of $k^\sigma[X^r]$, so that the centre of $\mathcal{R}_N$ is $E$.
\end{proof}
This result has various corollaries that are interesting for questions about factoring skew polynomials.
\begin{cor}\label{morita}
Let $N \in k^\sigma[X^r]$ be a nonzero polynomial that is not a power of $X$. Then $\mathcal{C}/N\mathcal{C}$ and $\mathcal{R}/N\mathcal{R}$ are Morita-equivalent.
\end{cor}
\begin{proof}
Since $\mathcal{R}$ is an Azumaya algebra over $\mathcal{C}$, $\mathcal{R}/N\mathcal{R}$ is an Azumaya algebra over $\mathcal{C}/N\mathcal{C}$.  Since $\mathcal{C}/N\mathcal{C}$ is a finite commutative ring, its Brauer group is trivial, hence $\mathcal{R}/N\mathcal{R}$ and $\mathcal{C}/N\mathcal{C}$ represent the same class in this Brauer group.
\end{proof}
\begin{cor}\label{matrixisomorphism}
Let $N \in k^\sigma[X^r]$ be an irreducible polynomial different from $X$. Let $E_N$ be the quotient field $\mathcal{C}/(N)$. Then
$$ \mathcal{R}/N\mathcal{R} \simeq \mathcal{M}_r(E_N),$$
the ring of $r \times r$ matrices with coefficients in $E_N$.
\end{cor}
\begin{proof}
By corollary $\ref{morita}$, $ \mathcal{R}/N\mathcal{R}$ is a ring of matrices with coefficients in $E_N$. The result follows from the fact that $\mathcal{R}/N\mathcal{R}$ has dimension $r^2\deg N$ over $k^\sigma$.
\end{proof}
One of the usual objects associated to Azumaya algebras is the notion of 
\emph{reduced norm}. This notion will be very important in the rest of 
the paper. In our situation, it is a multiplicative morphism 
$\mathcal{N}~:~k[X,\sigma][1/X] \rightarrow k^\sigma[X^r][1/X^r]$ which 
can be defined as follows. Consider the largest {\'e}tale subalgebra of 
$k[X,\sigma][1/X]$, which is $k[X^r][1/X^r]$. Then, $\mathcal N(x)$ is 
nothing but the determinant of the right-multiplication by $P$ on 
$k[X,\sigma][1/X]$ considered as a $k[X^r][1/X^r]$-module. Using that 
$k[X,\sigma]$ is a free module of rank $r$ over $k[X^r]$ (with basis 
$(1, X, \ldots, X^{r-1})$ for example), we deduce directly that $\mathcal N$ 
maps $k[X,\sigma]$ to $k^\sigma[X^r]$. We furthermore note that, if $P$ 
is the central skew polynomial (\emph{i.e.} $P \in k^\sigma[X^r]$), the 
(right-)multiplication by $P$ acts on a $k[X,\sigma]$ by 
(left-)multiplication by $P$ and therefore has determinant $P^r$. 
Therefore $\n(P) = P^r$ provided that $P \in k^\sigma[X^r]$.

\begin{rmq}
The property of being an Azymaya algebra could certainly be generalized to some other skew polynomal rings, for instance $k[X,\partial]$ where $\partial f = f' + f \partial$. For Azumaya algebras over rings whose Brauer group is trivial, many results of this paper should remain true. Since the triviality of the Brauer group is used strongly, there would probably be variations in the expected theorems when the Brauer group is nontrivial.
\end{rmq}
\subsection{Reinterpretation in terms of Galois representations}
In this section, we give a reinterpretation of the Morita equivalence in terms of Galois representations, recovering a variation of a theorem of Katz. 
%This approach was the original one that the autors took to study skew polynomials. 
Let us first give one definition.
\begin{defi}
A $\varphi$\emph{-module} over $k$ is a finite dimensional $k$-vector space $D$ endowed with an endomorphism $\varphi$ : $D \rightarrow D$ that is semilinear with respect to $\sigma$, \emph{i.e.} for all $x \in D$ and $a \in k$, $\varphi(\lambda x) = \sigma(\lambda) \varphi(x)$. A $\varphi$-module is said to be \emph{{\'e}tale} if the map $\varphi$ is injective.
\end{defi}
By definition, a $\varphi$-module (resp. an {\'e}tale $\varphi$-module) 
over $k$ is exactly a left-$k[X,\sigma]$-module having finite dimension 
over $k$.
\begin{defi}
If $P \in k[X,\sigma]$, the $\varphi$-module $D_P$ associated to $P$ is $k[X,\sigma]/k[X,\sigma]P$, endowed with the semilinear map $\varphi$ given by left-multiplication by $X$. We say that $P$ is {\'e}tale if $D_P$ is {\'e}tale. (It exactly means that the constant coefficient of $P$ is nonzero.)
\end{defi}
\begin{rmq}
Two skew polynomials $P$ and $Q$ are similar if and only if $D_P \simeq D_Q$.
\end{rmq}

The Morita equivalence shows the following:
\begin{cor}\label{moritakatz}
The category of {\'e}tale $\varphi$-modules over $k$ is equivalent to the category of finite dimensional $k^\sigma$-vector spaces endowed with an invertible endomorphism.
\end{cor}
\begin{proof}
Let $D$ be an {\'e}tale $\varphi$-module over $k$. Since $D$ has finite dimension over $k$, it is annihilated by some ideal $(N)$ of $\mathcal{C}$. By \ref{morita}, the categories of left-$\mathcal{R}/N\mathcal{R}$-modules and $\mathcal{C}/N\mathcal{C}$-modules are equivalent and we are done.
\end{proof}
This corollary can also be seen as a variation of the following theorem:
\begin{theo}[Katz]\label{katz}
Let $K$ be a field of characteristic $p>0$ endowed with a power of the Frobenius endomorphism $\sigma$. Then the category of {\'e}tale $\varphi$-modules over $K$ is equivalent to the category of $K^\sigma$-representations of the absolute Galois group of $K$.
\end{theo}
Indeed, if $\sigma(a) = a^{p^s}$, let $K = k\F_{p^s}$. Then $K^\sigma = k^\sigma$, and the absolute Galois group of $K$ is a procyclic group, so that a representation of this group is just the data of an invertible endomorphism of a $k^\sigma$-vector space of finite dimension (giving the action of a generator of the group).
The functor giving this equivalence is explicit: the representation corresponding to an {\'e}tale $\varphi$-module $D$ over $k$ is $\Hom_\varphi(D,K^\text{sep})$.
\begin{prop}\label{calculrep}
Let $(D,\varphi)$ be a $\varphi$-module over $k$, and let $\sigma^r$ be the generator of the absolute Galois group of $k \F_{p^s}$. Then the action of $\sigma^r$ on the $k^\sigma$-representation  $V$ corresponding to $D$ is isomorphic to $\varphi^r$:
$$(V \otimes_{k^\sigma} k, \sigma \otimes 1) \simeq (D, \varphi^r).$$
\end{prop}
\begin{proof}
It is enough to prove the result when $\varphi^r$ is cyclic. Let $f \in V = \Hom_\varphi(D,K^\text{sep})$. Then for $x \in D$, $\sigma^rf(x) = f(\varphi^r(x))$. This shows that the polynomials annihilating $\sigma^r$ and $\varphi^r$ are the same. The characteristic and minimal polynomials of $\sigma^r$ are the same, and equal to the characteristic polynomial of $\varphi^r$, so these two endomorphisms are conjugate.
\end{proof}
Using the fact that two skew polynomials are similar if and only if the corresponding $\varphi$-modules are isomorphic, we immediately get:
\begin{cor}\label{similarity}
Let $P,Q \in k[X,\sigma]$. The skew polynomials $P$ and $Q$ are similar if and only if the $k^\sigma[X^r]$-modules $(D_P, \varphi^r)$ and $(D_Q, \varphi^r)$ are isomorphic.
\end{cor}
Since $\varphi^r$ is a $k$-linear map, testing if these $k^\sigma[X^r]$ are isomorphic is completely straightforward.
\subsection{Factorizations}
In this section, we study some properties related to factorizations of skew polynomials and the structure of the corresponding $\varphi$-modules. First recall that if $P$ is a monic {\'e}tale skew polynomial, there is a bijection between all factorizations of $P$ as a product of monic irreducible skew polynomials, one the one hand, and all Jordan-H{\"o}lder sequence of the corresponding $\varphi$-module, on the other hand. By theorem \ref{calculrep}, these factorizations are also in bijection with Jordan-H{\"o}lder sequence of $D_P$ (viewed as a $k^\sigma[X^r]$-module). We shall see how we can use this to count the number of factorizations of $P$.

\subsubsection{Another definition of the norm}
Recall that we have defined the (reduced) norm of a skew polynomial $P \in k[X,\sigma]$ as the determinant of the right-multiplication by $P$ acting on the $k[X^r]$-module $k[X,\sigma]$. Proposition \ref{calculrep} allows us to give an equivalent definition:
\begin{lem}\label{norm-otherdef}
Let $P \in k[X,\sigma]$ be monic and let $(D_P,\varphi)$ be the corresponding $\varphi$-module. Then the norm $\n(P)$ is the characteristic polynomial of $\varphi^r$. If $P = a\tilde P$ with $\tilde P$ monic, then $\n(P) = N_{k/k^\sigma}(a) \cdot \n(\tilde P)$.
\end{lem}
\begin{proof}
Let $m_P$ be the right-multiplication by $P$ acting on $k[X,\sigma]$. Since both $P \mapsto \n(P) = \det m_P$ and $P \mapsto \chi_{\varphi^r}$ are multiplicative, it is enough to prove the Lemma when $P$ is monic irreducible. Let $\pi : k[X,\sigma] \to D_P$ be the canonical projection. We have $\pi \circ m_P = 0$. Since $\pi$ is surjective, the multiplication by $\det m_P$ is also  zero in $D_P$. This means that the minimal polynomial of the multiplication by $X^r$ on $D_P$ is a divisor of $\det m_P$. Since $P$ is irreducible, this minimal polynomial is the characteristic polynomial $\chi$ of $\varphi^r$. It is then enough to show (1)~that the degree of $\n(P)$ is the same as the degree of $\chi$ and (2)~that $\n(P)$ is monic. Write $P = P_0 + X P_1 + \cdots + X^{r-1}P_{r-1}$ with the $P_i$'s in $k[X^r]$. In the basis $(1, X, \ldots, X^{r-1})$, the matrix of $m_P$ is:
$$\begin{pmatrix}
P_0 & X^r\sigma(P_{r-1}) & \ldots & \ldots & X^r \sigma^{r-1}(P_1)\\
P_1 & \sigma(P_0) & \ddots & \ddots & \vdots \\
\vdots & \ddots & \ddots & \ddots & \vdots \\
\vdots & \ddots & \ddots & \ddots & X^r\sigma^{r-1}(P_{r-1})\\
P_{r-1} & \cdots & \cdots & \cdots & \sigma^{r-1}(P_0)
\end{pmatrix}.$$
Let $0 \leq i \leq r-1$ be the greatest integer such that the degree of $P_i$ is maximal, and denote by $\delta$ this degree. In the sum giving the determinant of this matrix, we have the term
$$P_i \sigma (P_i) \cdots \sigma^{r-i-1}(P_i) X^r\sigma^{r-i}(P_i) \cdots \sigma^{r-1}(P_i),$$
whose degree is $\delta(r-i) + (\delta +1)i = \delta r + i$ (as a polynomial in $X^r$). All the other terms of the determinant have degree less than this, so $\n(P) = \det m_P$ has degree $\delta r + i = \deg P = \deg \chi$ and is monic.
\end{proof}
\begin{prop}
Let $\n$ be the reduced norm map on $k[X,\sigma]$. Then the following properties hold:
\begin{itemize}
\item $\forall P \in k[X,\sigma]$, $P$ is a right- and left-divisor of $\n(P)$ in $k[X,\sigma]$,
\item $\forall P \in k[X,\sigma]$, $P$ is irreducible if and only if $\n(P)$ is irreducible in $k^\sigma[X^r]$,
\item If $P,Q \in k[X,\sigma]$ and $P$ is irreducible, then $P$ and $Q$ are similar if and only if $\n(P) = \n(Q)$ (up to multiplicative constant).
\end{itemize}
\end{prop}
\begin{proof}
The first fact is well-known (see for instance \cite{jac}, Proposition 1.7.1). It can be seen easily from the fact that if $(D_P, \varphi)$ is the $\varphi$-module associated to $P$, then $\n(P) (\varphi) = 0$. Indeed, the left-ideal $\{ R \in k[X,\sigma]~|~R(\varphi) = 0 \}$ is exactly $k[X,\sigma]P$.\\
For the second assertion, remark that $P$ is irreducible if and only if $D_P$ is simple, which holds if and only if the corresponding representation is irreducible. This is true if and only if the characteristic polynomial of $\varphi^r$ is irreducible in $k^\sigma[X^r]$.\\
Finally, we have already seen that the similarity class of a skew polynomial is determined by the conjugacy class of the action of $\varphi^r$ on the corresponding $\varphi$-module (corollary \ref{similarity}). For irreducible elements, this is completely determined by the characteristic polynomial of $\varphi^r$, \emph{i.e.} the reduced norm.
\end{proof}
Since $P$ is a divisor of $\n(P)$, we can expect that if $\tilde{N}$ is some irreducible factor of $\n(P)$ in $k^\sigma[X^r]$, then rgcd$(\tilde N, P)$ would be a nonconstant right-divisor of $P$. This is actually always true and formalized by the following lemma:
\begin{lem}\label{factornorm}
Let $P \in k[X,\sigma]$ be {\'e}tale and monic. Let $N = \n(P)$. If $N = N_1\cdots N_m$ with all $N_i$'s irreducible. Then there exist $P_1,\ldots, P_m \in k[X,\sigma]$ such that $P = P_1 \cdots P_m$ and for all $1\leq i \leq m$, $\n(P_i) = N_i$.\\
Moreover, $P_m$ can be chosen as an irreducible right-divisor of rgcd$(P,N_m)$.
\end{lem}
\begin{proof}
By induction on $m$, it is enough to prove the last assertion. Let $V_P$ 
be the Galois representation corresponding to the $\varphi$-module $D_P$ 
\emph{via} Katz's equivalence of categories (\emph{cf} Theorem 
\ref{katz}). Using Proposition \ref{calculrep}, we find that $V_P$ has a 
subrepresentation which is isomorphic to the quotient $k^\sigma[X^r] / 
N_m$ (where $\sigma^r$ acts by multiplication by $X^r$). Hence, there 
exists a surjective map $D_P \to D_{P_m}$ where $P_m$ is some skew polynomial 
of reduced norm $N_m$. It implies that $P_m$ is a right divisor $P$ and
then also a right divisor of rgcd$(P,N_m)$. This concludes the proof.
\end{proof}
\begin{rmq}
This result shows how to determine the similarity classes of irreducible skew polynomials appearing in a factorization of $P$. It also shows that any order is possible for the appearance of these similarity classes in a factorization of $P$.
\end{rmq}
When $\tilde N$ is an irreducible factor of $\n(P)$, the right greatest 
common divisor rgcd$(\tilde N, P)$ is never constant, so if we want to 
factor $P$ as a product of irreducible polynomials, we only need to know 
how to factor skew polynomials which are right-divisors of irreducible 
elements of the centre $k^\sigma[X^r]$.

\subsubsection{On the structure of $D_P$ when $P$ divides an irreducible central polynomial}

Let $N \in k^\sigma[X^r]$ be a monic irreducible polynomial, and let $E = k^\sigma[X^r]/(N)$. Let $P \in k[X,\sigma]$ be a right-divisor of $N$. 
The previous section has shown that factoring skew polynomials can be reduced to factoring skew polynomials of this form.
In this theoretical section, we begin a close study of the structure of $D_P$. All the results we are going to prove will play a very important role in the next section when we will be interested in designed a fast algorithm for factorization of skew polynomials.

We first remark that, since $\n(N) = N^r$, the norm of $P$ is $N^e$ for some integer $e \in \{1, \ldots, r\}$.
\begin{lem}
The $\varphi$-module $D_N$ is isomorphic to a direct sum of $r$ copies of a simple $\varphi$-module.
\end{lem}
\begin{proof}
It follows directly from Corollary \ref{morita}.
\end{proof}
The Lemma implies that if $P$ is a right-divisor of $N$ with $\n(P) = N^e$, then the $\varphi$-module $D_P = k[X,\sigma]/k[X,\sigma]P$ is isomorphic to a direct sum of $e$ copies of a simple $\varphi$-module. From this, we deduce that $\End_\varphi(D_P) \simeq \mathcal{M}_e(E)$.

\paragraph{Ring of endomorphisms}

From now on, we write $N = PQ$ for some $Q\in k[X,\sigma]$. Note that it implies that $QN = QPQ$; therefore $NQ = QPQ$ (since $N$ lies in the centre) and, simplifying by $Q$, we get $N = QP$. In other words $P$ and $Q$ commute.
The following proposition compares the $\varphi$-module $D_P = k[X,\sigma]/k[X,\sigma]P$ and its ring of endomorphisms.
\begin{prop}\label{surjectivity}
The map
$$ \begin{array}{ccl}
D_P & \rightarrow & \End_\varphi(D_P)\\
R & \mapsto & m_{QR}~:~\left|\begin{array}{ccc}
D_P &\rightarrow &D_P\\
x & \mapsto & xQR
\end{array}\right.
\end{array}
$$
is a surjective additive group homomorphism.
\end{prop}
Note that since $PQ = QP = N$ is central in $k[X,\sigma]$, the map above is well-defined. 
Indeed, we have to check that if $x \equiv x' \pmod P$ and $R \equiv R' \pmod P$ then $xQR \equiv x'QR' \pmod P$. Writing $x' = x + SP$ and $R' = R + TP$, we have:
\begin{eqnarray*}
x'QR' &=& xQR + SPQR + (xQT + SPQ)P \\
&\equiv &xQR + SNR \equiv xQR + SRN \equiv xQR + SRQP \equiv xQR \pmod P
\end{eqnarray*}
which is exactly what we want.
In order to prove the proposition, we will need the following lemma, that states that in the case $P = N$, that map is in fact an isomorphism.
\begin{lem}\label{surjectivityN}
Let $N \in k[X,\sigma]$. Then the map:
$$
\begin{array}{ccl}
D_N & \rightarrow & \End_\varphi(D_N)\\
R & \mapsto & m_R~:~\left|\begin{array}{ccc}
D_N &\rightarrow &D_N\\
x & \mapsto & xR
\end{array}\right.
\end{array}
$$
is an isomorphism of rings.
\end{lem}
\begin{proof}
The fact that our map is a morphism of rings is straightforward.
It is injective because $R = m_R(1)$. For the surjectivity, we remark that if $N$ is a commutative polynomial of degree $\delta$, $D_N$ has dimension $\delta r^2$ over $k^\sigma$ and, on the other hand, that if $E$ is the field $k^\sigma[X^r]/(N)$, $\End_\varphi(D_N)$ is isomorphic to $\mathcal{M}_r(E)$, so it also has dimension $\delta r^2$.
\end{proof}
\begin{proof}[Proof of Proposition \ref{surjectivity}]

We have the exact sequence of $\varphi$-modules:
$$ 0 \rightarrow k[X,\sigma]P/k[X,\sigma]N \rightarrow D_N \rightarrow D_P \rightarrow 0,$$
and $D_Q$ is isomorphic to $k[X,\sigma]P/k[X,\sigma]N$ \emph{via} the multiplication by $P$.
Since $D_N \simeq D_P^{\oplus r}$, this sequence is split. Let $s$ : $D_P \rightarrow D_N$ be a section. We have $Ps(1) = s(P) \equiv 0 \pmod{N}$, so there exists $S \in D_N$ such that $Ps(1) = NS$. Thus $s(1) = QS$. On the other hand, $QS = s(1) \equiv 1 \pmod{P}$. Hence there exists some $V \in k[X,\sigma]$ such that
$$ QS + VP = 1.$$
It implies that $D_P$ is isomorphic to $k[X,\sigma]QS/k[X,\sigma]N$ \emph{via} the multiplication by $QS$.

Let $u \in \End_\varphi(D_P)$, and let $A = u(1) \in D_P$. For all 
$x \in k[X,\sigma]$, $u(x) = xu(1) = xA$. In other words, $u$ is the 
mulitplication by $A$, \emph{i.e.} $u = m_A$. We then want to show that 
$m_A$ is of the form $m_{QR}$ for some $R \in D_P$. Let $\tilde u$ the 
endomorphism of $k[X,\sigma]QS/k[X,\sigma]N$ deduced from $u$: we have 
$\tilde u (QS) = AQS$.

Since $D_N = k[X,\sigma]QS/k[X,\sigma]N \oplus k[X,\sigma]P / k[X,\sigma]N$ (decomposition of $\varphi$-modules), we can extend $\tilde u$ to $D_N$ by setting $\tilde u (P) = 0$. By Lemma \ref{surjectivityN}, there exists $T \in D_N$ such that for all $x \in D_N$, $\tilde u (x) = xT$. In particular:
$$ \left\{ \begin{array}{ccc}
PT &\equiv&0 \pmod N\\
QST & \equiv & AQS \pmod N
\end{array}
\right.$$
Since $VPT + QST = T$, we have $QST \equiv T \pmod N$. So, for $x \in D_N$, we get $\tilde u (xQS) = xQST = xQSAQS = (xQSA)QS$. Hence, for $x \in D_P$, $u(x) = xQSA$. Setting $R = SA$, we have $u = m_{QR}$.
\end{proof}
\begin{cor}\label{random}
Let $R$ be a random variable uniformly distributed on $D_P$. Then the right multiplication by $QR$, $m_{QR}$, is uniformly distributed on $\End_\varphi(D_P) \simeq \mathcal{M}_e(E).$
\end{cor}
\begin{proof}
Since $R \mapsto m_{QR}$ is surjective, the probability that $m_{QR}$ is equal to $u \in \End_\varphi(D_P)$ is proportional to the cardinality of the fiber above $u$. We conclude the proof by remarking that $k$-linearity together with surjectivity implies that all fibers have the same cardinality.
\end{proof}

\paragraph{Some remarks about rgcd's and llcm's}

Let us present a rather elementary geometric point of view on rgcd's and llcm's in skew polynomial rings. If $P \in k[X,\sigma]$ is a divisor of an irreducible commutative polynomial $N$ of norm $N^e$, and $P_1$ is a right-divisor of $P$ of norm $N^{e_1}$, then $k[X,\sigma]P_1/k[X,\sigma]P \subset D_P$ is a sub-$E$-vector space $F_1$ of $D_P$ of dimension $e - e_1$.

If $P_2$ is another right-divisor of $P$ of norm $N^{e_2}$, it defines a sub-$E$-vector space $F_2$ of dimension $e - e_2$. 

The intersection and sum of these vector spaces have a description in terms of rgcd's and llcm's:
\begin{lem}
Let $R = \text{rgcd}(P_1,P_2)$ and let $M = \text{llcm}(P_1,P_2)$. Then: 
\begin{itemize}
\item $F_1 + F_2 = k[X,\sigma]R/k[X,\sigma]P$,
\item $F_1 \cap F_2 = k[X,\sigma]M/k[X,\sigma]P$.
\end{itemize}
\end{lem}
\begin{proof}
Left to the reader.
\end{proof}
\begin{rmq}
We will mainly use this Lemma when $P_1$ is irreducible. Then $k[X,\sigma]P_1/k[X,\sigma]P$ is an hyperplane in $D_P$. If we take the image of this hyperplane under any automorphism of $D_P$, we get another hyperplane, and it is likely that the intersection of this hyperplane with $k[X,\sigma]P_1/k[X,\sigma]P$ has codimension $2$ in $D_P$, and hence it is an hyperplane in $k[X,\sigma]P_1/k[X,\sigma]P$. We get this way an irreducible divisor of the quotient of the right division of $P$ by $P_1$.
\end{rmq}
\subsubsection{Counting factorizations}
In this section, we explain how to compute the number of factorizations of a monic skew polynomial $P \in k[X,\sigma]$ as a product of monic irreducible polynomials.
\begin{lem}
Let $P \in k[X,\sigma]$ be a monic {\'e}tale skew polynomial. Assume that $\n(P) = N^e$ with $N$ irreducible of degree $d$, and $P$ is a right-divisor of $N$. Then the number of factorizations of $P$ as a product of monic irreducible skew polynomials is the $q^d$-factorial $[e]_{q^d}! = \frac{(q^{de}-1)\cdots (q^d-1)}{(q^d -1)^e} $ 
\end{lem}
\begin{proof}
By induction on $e$, it is enough to prove that $P$ has exactly $\frac{q^{de} -1}{q^d -1}$ monic irreducible right-divisors. The number of monic irreducible right-divisors is also the number of simple sub-$\varphi$-modules of $k[X,\sigma]/k[X,\sigma]P$. Such submodules are in bijection with $k^\sigma[X^r]/(N)$-lines in $k[X,\sigma]/k[X,\sigma]P$ (if $P_1$ is an irreducible right-divisor of $P$, every irreducible right-divisor of $P$ can be written as the image of $P_0$ by an endomorphism of $k[X,\sigma]/k[X,\sigma]P$), so it has the cardinality of the projective space $\mathbf{P}(E^e)$, which is $\frac{q^{de} -1}{q^d -1}$.
\end{proof}

Let us now define the \emph{type} of a skew polynomial $P$. Recall first that the endomorphism $\varphi^r$ of $D_P$ defined over $k^\sigma$ can be put into Jordan form, and that a Jordan block is by definition  an invariant subspace in a basis of which the restriction of $\varphi^r$ has a matrix of the form:
$$\begin{pmatrix}
A & I & 0 & \cdots & 0\\
0 & A & I & \ddots & \vdots\\
\vdots & \ddots & \ddots & \ddots & \vdots\\
\vdots & \ddots & \ddots & \ddots & I\\
0 &\cdots & \cdots  & 0 & A
\end{pmatrix},
 $$
 with $A$ having a characteristic polynomial that is irreducible in $k^\sigma[X^r]$. We will say that a Jordan block has size $i$ if the number of matrices $A$ that appear in this block is $i$.
\begin{defi}
Let $P \in k[X,\sigma]$. Assume that $\n(P) = N^e$ for some integer $e$. For $i \geq 1$, let $e_i$ be the number of Jordan blocks of $(D_P, \varphi^r)$ of size at least $i$. Let $n$ be the largest index such that $e_i \neq 0$, we say that $P$ has type $(e_1, \ldots, e_n)$.
\end{defi}
This means that if $P$ has type $(e_1, \ldots, e_n)$, the action of $\sigma^r$ on the corresponding representation has exactly $e_1$ Jordan blocks, $e_2$ of which contain a block of the form $\begin{pmatrix} A & I\\ 0 & A \end{pmatrix},$ etc.
\begin{rmq}\label{dualdiagram}
The type is determined by the nonincreasing sequence $(a_1, \ldots, a_m)$, $a_i$ being the size of the $i$-th largest Jordan block of $(D_p, \varphi^r)$. The Young diagram associated to $(a_1, \ldots, a_m)$ is dual to the one associated to $(e_1, \ldots, e_n)$. We will say that $(a_1, \ldots, a_m)$ is the \emph{dual} sequence of $(e_1, \ldots, e_n)$. We also say that $(a_1, \ldots, a_m)$ is the \emph{dual type} of $P$.
\end{rmq}
A skew polynomial $P \in k[X,\sigma]$ has type $(e)$ (with $1 \leq e \leq r$) if and only if it is a divisor of an irreducible polynomial $N \in k[X,\sigma]$.
\begin{defi}
Let $P \in k[X,\sigma]$ be a monic skew polynomial. Let $N_1^{a_1}\cdots N_t^{a_t}$ be the factorization of $\n(P)$ as a product of monic irreducible polynomials. For $1 \leq i \leq t$, let $(e_1^{(i)}, \ldots, e_{n_i}^{(i)})$ be the type of the restriction of $\varphi^r$ to the characteristic invariant subspace associated to $N_i$. We say that $P$ has type:
$$(N_1, (e_1^{(1)}), \ldots, e_{n_1}^{(1)}) , \ldots, (N_t, (e_1^{(t)}, \ldots, e_{n_t}^{(t)})).$$
\end{defi}
As we have seen before, if $P \in k[X,\sigma]$ is a monic polyomial, then the set of all factorizations of $P$ as a product of monic irreducible polynomials is in bijection with the Jordan-H{\"o}lder sequences for the $\varphi$-module $D_P$. It is also in bijection with all bases of $D_P$ in which $\varphi^r$ has Jordan form. This can be described in terms of the types of some factors of $P$, as we are going to explain now.

First, we assume that $\n(P) = N^e$ with $N \in k^\sigma[X ^r]$ irreducible. Let $(e_1, \ldots, e_n)$ be the type of $P$. We denote by $V$ the representation associated to $D_P$, with $g$ the endomorphism through which $\sigma^r$ acts on $V$. If $W$ is any nonzero irreducible invariant subspace of $V$, the action of $g$ on $W$ is given by the companion matrix of $N$ in some basis. Let $(a_1, \ldots, a_m)$ be the dual sequence of $(e_1, \ldots, e_n)$ as defined in Remark \ref{dualdiagram}.
\begin{lem}\label{pathweight}
Let $\delta = \deg N$. Let $1 \leq i \leq m$ such that $i=m$ or $a_i> a_{i+1}$. Let $i_0$ be the smallest $j$ such that $a_j = a_i$. Then there are $q^{\delta(i-1)} + q^{\delta i} + \cdots + q^{\delta(i_0-1)}$ invariant irreducible subspaces $V'$ of $V$ such that the quotient $V/V'$ has a type whose dual is $(a_1, \ldots, a_i -1, a_{i+1}, \ldots, a_m)$ (or $(a_1, \ldots,a_{m-1})$ if $i=m$ and $a_m = 1$).
\end{lem}
\begin{proof}
Denote by $(\varepsilon_{1,1},\ldots,\varepsilon_{1,\delta}, \varepsilon_{2,1},\ldots, \varepsilon_{2,\delta}, \ldots)$ a basis of $V$ in which the matrix of $g$ has Jordan form. More precisely, for all $1\leq i \leq m$, and for all $1\leq j \leq t_i$ and $1\leq l \leq \delta$, we have $g(\varepsilon_{j,l}) = \varepsilon_{j,l+1}$ if $(j,l)$ is not of the shape $(j,1)$ for some integer $j \geq 2$, or of the shape $(j,\delta)$ for some integer $j \geq 1$, $g(\varepsilon_{j, 1}) = \varepsilon_{\delta u,\delta} + \varepsilon_{\delta u +1, 2}$ if $j \geq 2$, and $g(\varepsilon_{j,\delta}) = \sum_{l=1}^\delta \alpha_{l}e_{j,l}$, where $\sum_{l=1}^{\delta} \alpha_l X^{r(l-1)}= N$ (it is the characteristic polynomial of the induced endomorphism on any irreducible invariant subspace).\\
There are $i_0 - 1$ Jordan blocks of $g$ whose length is greater than the length of the $i$-th block. For $\lambda = (\lambda_{1,1}, \ldots, \lambda_{1,\delta},\ldots, \lambda_{i_0-1,\delta}) \in {k^\sigma}^{\delta(i_0 -1)}$, let $v_\lambda = e_{i_0,1} + \sum_{j=1}^{i_0 -1} \sum_{l = 1}^\delta \lambda_{j,l} e_{j,l}$. Since two such vectors $v_\lambda$, $v_\mu$ are not colinear, they generate distinct invariant subspaces $V_\lambda$, $V_\mu$, which are clearly isomorphic to $W$. Moreover, the quotient $V/V_\lambda$ has the same type as $V/V_{(0)}$ because the map $V \rightarrow V$ that sends $\varepsilon_{i_0,1}$ to $v_\lambda$ and is the identity outside the invariant subspace generated by $\varepsilon_{i_0,1}$ is an isomorphism (its matrix is upper triangular). One can build the same way invariant subspaces with quotients of the same type as generated by vectors of the shape $\varepsilon_{i_0+1,1} + \sum_{j=1}^{i_0 }  \sum_{l = 1}^\delta \lambda_{j,l} \varepsilon_{j,l} , \ldots, \varepsilon_{i,1} + \sum_{j=1}^{i -1}  \sum_{l = 1}^\delta \lambda_{j,l} \varepsilon_{j,l}$. There are exactly $q^{\delta i_0-1} + \cdots + q^{\delta i-1}$ invariant subspaces that are built in this way. Doing such constructions for each $i'$ satisfying the hypotheses of the lemma, we get exactly $\frac{q^{\delta m} -1}{q^\delta -1}$ irreducible invariant subspaces, which means all of them. Among these subspaces, the ones for which the quotient has the requested shape are exactly the  $q^{\delta(i_0-1)} + \cdots + q^{\delta(i-1)}$ built for the first $i$ we considered. This proves the lemma.
\end{proof}
In order to compute the number of Jordan-H{\"o}lder sequences of $g$, consider the following diagram:
$$\begin{tabular}{|c|c|c|c|}
\multicolumn{1}{c}{1} & \multicolumn{1}{c}{$q^\delta$} & \multicolumn{1}{c}{$\ldots$} & \multicolumn{1}{c}{$q^{\delta(m-1)}$}\\
\hline
$a_1$ & $a_2$ & $\ldots$ & $a_{m}$\\
\hline
\end{tabular}$$
with $a_1\geq \ldots \geq a_m$. An \emph{admissible path} is a transformation of this table into another table  $\begin{tabular}[b]{|c|c|c|c|}
\multicolumn{1}{c}{1} & \multicolumn{1}{c}{$q^\delta$} & \multicolumn{1}{c}{$\ldots$} & \multicolumn{1}{c}{$q^{\delta(m'-1)}$}\\
\hline
$a_1'$ & $a_2'$ & $\ldots$ & $a_{m'}'$\\
\hline
\end{tabular}$ such that 
\begin{itemize}
\item either $m' = m-1$, $a_j' = a_j$ for $1 \leq j \leq m-1$, if $a_m = 1$;
\item or $m' = m$, $a_j' = a_j$ for all $j\neq i$, with $1 \leq i \leq m$ such that $a_{i} > a_{i+1}$.
\end{itemize}
To such a path $\gamma$, we affect a weight $w(\gamma)$, which is the sum of the coefficients written above the cells of the first table containing the same number $a_i$ as the cell whose coefficient was lowered in the second table. Here is an example of a table and all the admissible paths with the corresponding weights:
$$
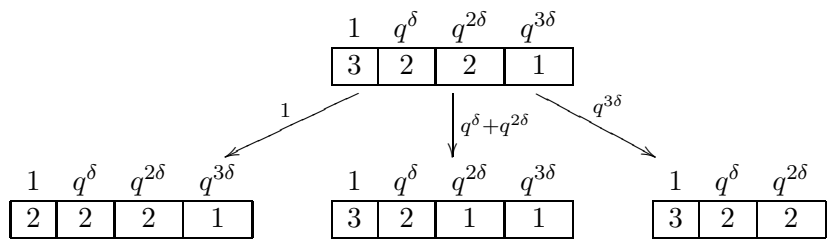
\xymatrix{
& {\begin{tabular}{|c|c|c|c|}
\multicolumn{1}{c}{1} & \multicolumn{1}{c}{$q^\delta$} & \multicolumn{1}{c}{$q^{2\delta}$} & \multicolumn{1}{c}{$q^{3\delta}$}\\
\hline
$3$ & $2$ & $2$ & $1$\\
\hline
\end{tabular}} \ar[dl]_{1} \ar[d]^{q^\delta + q^{2\delta}} \ar[dr]^{q^{3\delta}}\\
{\begin{tabular}{|c|c|c|c|}
\multicolumn{1}{c}{1} & \multicolumn{1}{c}{$q^\delta$} & \multicolumn{1}{c}{$q^{2\delta}$} & \multicolumn{1}{c}{$q^{3\delta}$}\\
\hline
$2$ & $2$ & $2$ & $1$\\
\hline
\end{tabular}}
&
{\begin{tabular}{|c|c|c|c|}
\multicolumn{1}{c}{1} & \multicolumn{1}{c}{$q^\delta$} & \multicolumn{1}{c}{$q^{2\delta}$} & \multicolumn{1}{c}{$q^{3\delta}$}\\
\hline
$3$ & $2$ & $1$ & $1$\\
\hline
\end{tabular}}
&
{\begin{tabular}{|c|c|c|}
\multicolumn{1}{c}{1} & \multicolumn{1}{c}{$q^\delta$} & \multicolumn{1}{c}{$q^{2\delta}$} \\
\hline
$3$ & $2$ & $2$\\
\hline
\end{tabular}}
}$$
By lemma \ref{pathweight}, the weight of an admissible path from one table to another, is the number of irreducible invariant subspaces of an endomorphism $g$ with type whose dual is given by the first table such that the quotient has the type given by the second table. Therefore, a sequence of admissible paths ending to an empty table represents a class of Jordan-H{\"o}lder sequences. Thus the number of distinct sequences in this class is the product of the weights of the paths along the sequence. Hence, the number of Jordan-H{\"o}lder sequences for $g$ is $\sum_{(\gamma_1,\ldots \gamma_\tau)} \prod_{i=1}^{\tau} w(\gamma_i)$ the sum being taken on all sequences $(\gamma_1, \ldots, \gamma_\tau)$ of admissible paths ending at the empty table (so $\tau = \sum_{j=1}^m a_j$).
\begin{cor}
Let $P \in k[X,\sigma]$ be a monic {\'e}tale polynomial of dual type $(a_1, \ldots, a_m)$. Then the number of factorizations of $P$ as a product of monic irreducible polynomials is
$$\sum_{(\gamma_1,\ldots \gamma_\tau) \text{ admissible}} \prod_{i=1}^{\tau} w(\gamma_i).$$
\end{cor}
\begin{ex}
If $P$ has type $(e)$ (so that its dual type is $(1, \ldots, 1)$), then there is only one admissible path, and the number of factorizations is $[q^\delta]_e = \prod_{i=1}^e \frac{q^{\delta i} -1}{q^\delta -1}$.
\end{ex}
\begin{ex}\label{indecomposable}
If $P$ has type $(1, \ldots, 1)$, (so that its dual type is $(a)$), there is only one admissible path, and only one factorization. The formula also shows that polynomials of this type are the only ones that have a unique factorization. These polynomials have already been studied for their interesting properties, under the name of \emph{lclm-indecomposable} (see \cite{jac2}, Chap. 3, Th. 21 and 24 for properties, and \cite{felix2} for applications). 
\end{ex}
For the general case, there is no such nice formula, but we can still explain how to get the number of factorizations. By the Chinese remainders Theorem, $V$ is a direct sum of invariant subspaces on which the induced endomorphisms have minimal polynomial that is a power of an irreducible. Here, the type of $g$ is defined again as the data of $((W_1,T_1), \ldots, (W_s,T_s))$ where the $W_l$'s are the distinct classes of irreducible invariant subspaces of $V$, and the $T_l$'s are the tables representing the types of the endomorphisms induced on the corresponding subspaces of $V$. The notion of dual type can be defined as previously, as well as the notion of admissible path.
\begin{prop}
Let $g$ be an endomorphism of an $\F_q$-vector space $V$. Assume that the dual type of $g$ is $((W_1,T_1), \ldots, (W_s,T_s))$. Denote by $\delta_i$ the dimension of $W_i$, and by $ \tau_i$ the sum of the coefficients in table $T_i$. Then the number of Jordan-H{\"o}lder sequences of $g$ is 
$$\frac{(\tau_1 + \cdots + \tau_s)!}{\tau_1!\cdots \tau_s!} \prod_{(\Gamma_1, \ldots, \Gamma_s)} w(\Gamma_1)\cdots w(\Gamma_s),$$
the product being taken over all the $s$-uples $(\Gamma_1, \ldots, \Gamma_s)$ of admissible path sequences ending at the empty tables.
\end{prop}
\begin{proof}
From a chain of admissible paths ending at $((W_1, \emptyset), \ldots, (W_s, \emptyset))$, it is possible to extract its $W_l$-part $\Gamma_l$ for all $1 \leq l \leq s$. By definition, it is the sequence of all the paths involving a change in the table associated to $W_l$. Such a chain is a sequence of admissible paths from $T_l$ ending at the empty table. It is clear that the weight of the path sequence is the product of the weights of the $\Gamma_l$'s. Therefore, it does not depend on the way the $\Gamma_l$'s were combined together. The admissible path sequences that end at $((W_1, \emptyset), \ldots, (W_s, \emptyset))$ are all the different ways to recombine admissible path sequences from all the $(W_i,T_i)$ to the empty table. The weight of such a sequence is the product of the weights of the $W_l$-parts. There are as many recombinations as anagrams of a word that includes $\tau_l$ times the letter $W_l$ for all $1 \leq l \leq s$, $\tau_l$ being the sum of the integers appearing in $T_l$. The result then follows directly from the previous discussion an the fact that the number of anagrams of a word that includes $\tau_l$ times the letter $W_l$ is the multinomial coefficient $\frac{(\tau_1 + \cdots + \tau_s)!}{\tau_1!\ldots \tau_s!} $.
\end{proof}
\begin{ex}
Assume $g$ has dual type $((W_1, (a_1)) , \ldots, (W_s, (a_s)))$. It is easy to see that the only admissible path sequence for $(W_l,(t_l))$ has weight 1. Hence the number of Jordan-H{\"o}lder sequences of $g$ is $\frac{(a_1 + \cdots + a_s)!}{a_1!\ldots a_s!}$. This generalizes remark \ref{indecomposable}
\end{ex}
\begin{cor}\label{countfac}
Let $P \in k[X,\sigma]$ monic {\'e}tale. Let $((W_1,T_1), \ldots, (W_s,T_s))$ be the type of $P$. Then the number of factorizations of $P$ as a product of monic irreducible polynomials is
$$\frac{(\tau_1 + \cdots + \tau_s)!}{\tau_1!\cdots \tau_s!} \prod_{(\Gamma_1, \ldots, \Gamma_s)} w(\Gamma_1)\cdots w(\Gamma_s),$$
the product being taken over all the $s$-uples $(\Gamma_1, \ldots, \Gamma_s)$ of admissible path sequences ending at the empty tables.
\end{cor}
In the next section, we describe an algorithm for counting the number of factorizations of a skew polynomial relying on this theory.
\begin{rmq}
If $\n(P)$ is a power of an irreducible commutative polynomial $N$, say $\n(P) = N^a$, then the type of $P$ can be determined as follows: let $P_1 = P$, and for $i \geq 1$, define $Q_1 = \text{rgcd}(P_i,N)$ and $P_{i} = P_{i+1}Q_i$. Let $m$ be minimal such that $Q_{m+1} = 1$. Then for all $1 \leq i \leq m$ , $\n(Q_i) = N^{e_i}$ for some integer $1 \leq e_i \leq  r$, and the type of $P$ is $(e_1, \ldots, e_m)$. The type can also be determined by looking at the degrees of the successive rgcd's of $P$ with $N, N^2, N^3, \ldots$
\end{rmq}

\section{Computational aspects}

This section deals with several computational aspects of skew polynomial rings. In the first part, we describe algorithms for arithmetics in these rings: multiplication, Euclidean division, gcd's and lcm's, and we give their complexities. Then, we give algorithms to compute the reduced norm of a skew polynomial as defined in the theoretical part. We use these algorithms and some other theoretical results to give a fast factorization algorithm. We give a detailed computation of the complexity of this algorithm. Finally, we describe algorithms for factorization-counting and random factorizations.

Throughout this section, we will use the following notations:
\begin{itemize}
\item $\MM(n)$ is the number of operations (in $k^\sigma$) needed to compute the product of two $n \times n$ matrices with coefficients in $k^\sigma$.
\item $\SM(n,r)$ is the number of operations (in $k^\sigma$) needed to multiply two skew polynomials with coefficients in $k$ of degree at most $n$.
\end{itemize}

We recall that we have proved in \S \ref{fastmult} that one can take 
$\SM(n,r) = \tilde O(n r^2)$. Regarding matrix multiplication, the naive 
algorithm gives $\MM(n) = O(n^3)$ but it is well known that this 
complexity can be improved. For instance, using Strassen's algorithm, 
one have $\MM(n) = O(n^{\log_2 7})$. Today, the best known asymptotic 
complexity for matrix multiplication is due to Vassilevska Williams 
\cite{vw} and is about $O(n^{2.3727})$.

We use the common $\tilde O$ notation: if $f$ and $g$ are two real functions defined on the integers, we say that $f(n) = \tilde O(g(n))$ if there is some integer $m$ such that $f(n) = O(g(n)\log^m(n))$.

We also assume that all usual arithmetics with polynomials can be done 
in quasilinear time. In particular, we assume that all usual operations 
(basically addition, multiplication and inverse) in an extension of 
$k^\sigma$ of degree $d$ requires $\tilde O(d)$ operations in 
$k^\sigma$. We refer to \cite{mca} for a presentation of algorithms having 
these complexities. Regarding the Frobenius morphism on $k$, we assume 
that all the conjugates of an element $a \in k$ can be computed in 
$O(r^2)$ operations in $k^\sigma$.

\subsection{Fast arithmetics in skew polynomial rings}
This section is dedicated to basic algorithms for arithmetics in skew polynomial rings.
\subsubsection{Multiplication}\label{fastmult}
Let $A,B \in k[X,\sigma]$, both of degree $\leq d$. We give several algorithms to compute the product $AB$ and we compare their complexities.

\paragraph{The classical algorithm}

Let us recall that the classical algorithm of \cite{gie2}, Lemma 1.1 
(which throughout this section will be referred to as ``Giesbrecht's 
algorithm'') has complexity $\tilde{O}(d^2r + dr^2)$.
This algorithm uses the explicit formula for the coefficients of the 
product of two skew polynomials: if $A = \sum_{i=0}^{d_1} a_i X^i$ and 
$B = \sum_{i=0}^{d_2} b_i X^j$, then their product is $\sum_{i=0}^{d_1 + 
d_2} \left(\sum_{j=0}^i a_j\sigma^j(b_{i-j}) \right)X^i$. For each 
coefficient $b_i$ of $B$, the list of the images of $b_i$ under all the 
powers of $\sigma$ can be computed in $O(r^2)$ operations in $k^\sigma$. 
Hence, all the $\sigma^j(b_{i-j})$ that may appear in the above formula 
can be computed in $\tilde{O}(d_2r^2)$. Once we have these coefficients, 
it remains to compute the product, which is done with $O(d_1 d_2)$ 
operations in $k$, so the total complexity is $\tilde{O}(d_2r^2 + d_1 
d_2 r)$. To write it more simply, if both polynomials have degree less 
than $d$, then their product can be computed in $\tilde{O}(d^2r + dr^2)$ 
operations in $k^\sigma$.

\paragraph{Reduction to the commutative case}
Here, we use fast multiplication for commutative polynomials to multiply skew polynomials. Write $A = \sum_{i=0}^{r-1} A_i X^i$, with each $A_i$ in $k[X^r]$. For $0 \leq i \leq r-1$, we denote by $B^{(i)}$ the skew polynomial deduced from $B$ by applying $\sigma^i$ to all coefficients. Then we have:
$$ AB = \sum_{i=0}^{r-1} A_iB^{(i)}X^i.$$
Since $A_i \in k[X^r]$, it is easy to see that the product $A_iB^{(i)}$ is the same as the product of these polynomials computed in $k[X]$. The algorithm is the following:
\begin{enumerate}
\item Compute the $B^{(i)}$.
\item Compute all the products $A_iB^{(i)}$.
\item Compute the sum $AB =  \sum_{i=0}^{r-1} A_iB^{(i)}X^i$.
\end{enumerate}
\begin{lem}
The number of operations needed in $k^\sigma$ for the multiplication of two skew polynomials of degree at most $d$ by the above algorithm is $\tilde{O}(dr^2)$.
\end{lem}
\begin{proof}
We may assume that both $A$ and $B$ have degree $d$. For step $1.$, we need to compute all the conjugates of the $d$ coefficients of $B$, which can be done in $O(dr^2)$ operations in $k^\sigma$. The multiplications of step $2.$ as multiplications of elements of $k[X]$ can be done in $\tilde{O}(d)$ multiplications of elements of $k$, which corresponds to $\tilde O(dr)$ operations in $k^\sigma$. The total complexity of this step is then $\tilde{O}(dr^2)$ operations in $k^\sigma$. Finally, there are less than $2dr$ additions of elements of $k$ to do in step $3.$, which is done in $O(dr^2)$. The global complexity is therefore $\tilde{O}(dr^2)$.
\end{proof}

\begin{rmq}
This complexity is apparently better than those of Giesbrecht's 
algorithm (the term $dr^2$ has gone) but we want to note that 
Giesbrecht's algorithm can beat this ``commutative method'' if the
degree of $B$ is much more less than the degree of $A$. Indeed, in 
that case the dominant term in Giesbrecht's complexity is $d_1 d_2 r$
which can be competitive with $d_1 r^2$ if $r$ is large compared to 
$d_2$.
\end{rmq}

\bigskip

We are now going to present two variants of Karatsuba's multiplication
to the noncommutative case. Actually, it will turn out that the resulting
algorithms are asymptotically slower than the ``commutative method'';
nevertheless, we believe that they can be better in some cases and, for
this reason, we include them in this paper.

\paragraph{The plain Karatsuba method}
Let $A, B \in k[X,\sigma]$. Write $A = A_0 + X^{mr} A_1$ and $B = B_0 + X^{mr} B_1$, with $m = \lfloor \frac{\max\{\deg A,\deg B\}}{2r} \rfloor$. We can then write:
$$ AB = C_0 + X^{mr}C_1 + X^{2mr} C_2,$$
with $C_0 = A_0 B_0$, $C_1 = A_0 B_1 + A_1 B_0$ and $C_2 = A_1 B_1$, because $X^{mr}$ lies in the center of $k[X,\sigma]$. If we set $P = (A_0 + A_1)(B_0 + B_1)$, we get the fact that $C_1 = P - C_0 - C_2$. Hence, we can recover the product $AB$ doing the $3$ multiplications $C_0 = A_0B_0$, $C_2 = A_1 B_1$ and $P = (A_0 + A_1)(B_0 + B_1)$. Let $\text{MS}(d)$ be the number of multiplications needed in $k^\sigma$ to multiply two elements of $k[X,\sigma]$ of degree $\leq d$ using this method. We get:
$$ \SM(d,r) \leq 3\cdot \SM\left(\frac{d}{2}\right) \leq 3^\frac{\log (d/r)}{ \log 2} \cdot \SM(r).$$
Hence, this method allows to multiply polynomials of degree $\leq d$ in 
time $O\Big((\frac{d}{r})^{\frac{\log 3}{\log 2}} \SM(r)\Big)$ provided
that $d > r$.
Using Giesbrecht's algorithm for multiplication of skew polynomials of 
degree $<r$, we get a complexity of $O\left( d^{\frac{\log 3}{\log 
2}}r^{3 - \frac{\log 3}{\log 2}}\right)$, which is around $O 
\left(d^{1.58}r^{1.41} \right)$.

\paragraph{The Karatsuba-and-matrix method}
The previous Karatsuba method relies on the classical multiplication for polynomials of degree $<r$. Here, we propose another fast multiplication method for polynomials of degree up to $r^2/2$, that can be combined with the Karatsuba method.

Let $N \in k^\sigma[X^r]$ be the defining polynomial of the extension $k/k^\sigma$. We will denote by $t$ a root of $N$ in $k$. By Lemma \ref{matrixisomorphism}, the $\varphi$-module $k[X,\sigma]/N$ is isomorphic to $\mathcal{M}_r(k)$. Here, the isomorphism can be given explicitely. Indeed, the isomorphism of this Lemma maps $A \in k[X,\sigma]/N$ to the matrix of the right multiplication by $A$ in some basis of the $k$-vector space $k[X,\sigma]/N$. Let us choose the basis $1, X, \ldots, X^{r-1}$. If $a \in k$, the matrix of the right multiplication by $a$ is given by:
$$M_a = 
\begin{pmatrix}
a & 0 & \cdots & 0\\
0 & \sigma(a)& \ddots & \vdots\\
\vdots & \ddots & \ddots & \vdots\\
0 & \cdots & 0 & \sigma^{r-1}(a)
\end{pmatrix}.
$$
The matrix of multiplication by $X$ is:
$$
M_X = 
\begin{pmatrix}
0 & 1 & \cdots & 0\\
0 & \ddots& \ddots & 0\\
\vdots & \ddots & \ddots & 1\\
t & \cdots & 0 & 0
\end{pmatrix}.
$$
So, if $A= \sum_{i=0}^{r^2 -1} a_i X^i \in k[X,\sigma]$, the image of $A \pmod{N}$ can be computed easily by the previous isomorphism. More precisely, write $A = \sum_{i=0}^{r-1} A_i X^i$ with $A_i \in k[X^r]$ of degree $<r$. Then the matrix of $m_A$ is:
$$
M_A = 
\begin{pmatrix}
A_0(t) & \sigma(A_{r-1})(t) & \cdots & \sigma^{r-1}(A_1)(t)\\
tA_1(t) & \sigma(A_0)(t)& \cdots & \vdots\\
\vdots & \ddots & \ddots & \sigma^{r-1}(A_{r-1})(t)\\
tA_{r-1}(t) & \cdots & t\sigma^{r-2}(A_1)(t) & \sigma^{r-1}(A_{0})(t)
\end{pmatrix}.
$$

The matrix $M_A$ can be computed as follows. We first evaluate all the 
polynomials $A_i$'s at all the conjugates of $t$. Using efficient 
algorithms (see \cite{mca}, \S 10), it requires $\tilde{O}(r^3)$ 
operations in $k^\sigma$. We can then compute $\sigma^j(A_i)(t)$ by 
applying $\sigma^j$ to $A_i(\sigma^{-j}(t))$. Computing all these 
quantites requires $\tilde O(r^4)$ further operations in $k^\sigma$. 
Then to obtain $M_A$, it remains to multiply some of the previous 
coefficients by $t$, which requires at most $\tilde O(r^3)$ further 
operations in $k^\sigma$. Computing $M_A$ can then be done with 
complexity $\tilde O(r^4)$.

We can go in the other direction following the same ides. We first 
divide by $t$ all coefficients above the diagonal of $M_A$. We then 
apply $\sigma^{0}$ to the first column, $\sigma^{r-1}$ to the second 
column, $\ldots$, $\sigma$ to the last column and, finally, recover the 
$A_i$'s by interpolation. As before, the complexity is $\tilde O(r^4)$
operations in $k^\sigma$.

Once noticed these facts, the idea is quite simple: let $A,B \in 
k[X,\sigma]$ of degree $<r^2/2$. We compute the corresponding matrices 
$M_A, M_B$, then the product $M_AM_B$ and finally recover the 
coefficients of (the reduction modulo $N$) of $AB$. This whole algorithm 
can be done in $\tilde{O}(r^4)$ operations in $k^\sigma$.

Combining this with Karatsuba multiplication (but using this as soon as we hit polynomials of degree $<r^2/2$), we get:
$$\textstyle \SM(d,r) = \tilde{O}\big((\frac{d}{r^2})^\frac{\log 3}{\log 2}\cdot \SM(\frac{r^2}2,r)\big) = \tilde{O}(d^\frac{\log 3}{\log 2} r^{4 - \frac{2 \log 3}{\log 2}})$$
provided that $d > r^2$.
This is about $\tilde{O}(d^{1.58}r^{0.83})$.

\begin{rmq}
The most expensive step of the previous algorithm is the application
of the Frobenius. Hence, if we are working over a finite field where
applying Frobenius can be done efficiently, our complexity may decrease
to $O(r \cdot \MM(r))$ --- which beats the ``commutative method''. If 
we take $MM(r) = O(r^{2.3727})$, the resulting final complexity becomes 
$\tilde O(d^{1.58}r^{0.2})$.
\end{rmq}

%\paragraph{Complexity comparison}
%We will now compare the complexities of the various algorithms above, depending on the relative values of $d$ and $r$. For this purpose, let us denote by $\theta = \frac{\log 3}{\log 2}$ since this constant often appears in our calculations. Consider the problem of multiplying two skew polynomials of degree at most $d$ in $k[X,\sigma]$ with $[k:k^\sigma] = r$. Now let $\alpha = \frac{log d}{\log r}$. The  previous complexities are summed up in the following table. We give the value of $\beta$ such that the complexity is $\tilde{O}(r^\beta)$. Except for the commutative method, which is better than all the others in every range of $\alpha$ (but uses FFT, so it could become efficient only with $d$ large), each algorithm has a range of values of $\alpha$ for which it is more efficient than the others; we also indicate that range in the table.
%$$
%\begin{array}{|c|c|c|c|}
%\hline
%\text{Algorithm} & \beta  & \text{Best in range} & \text{Restriction}\\
%\hline
%\hline
%\text{Giesbrecht} & \max(\alpha +2,2\alpha +1) & \alpha \leq 1 & \text{None}\\
%\hline
%\text{Karastuba} & (\alpha -1)\theta + 3 & \alpha \in \left(1,2\right] & \alpha > 1\\
%\hline
%\text{Karatsuba + matrix} & (\alpha -2) \theta + 4 & \alpha \in \left(2, \frac{2\theta -1}{\theta -1}\right] & \alpha > 2\\
%\hline
%\hline
%\text{Commutative} & \alpha + 2 & \text{All}  & \text{None}\\
%\hline
%\end{array}
%$$

\subsubsection{Euclidean division}
Let $A,B \in k[X,\sigma]$ with $\deg A \geq \deg B$. We want to compute the right-Euclidean division of $A$ by $B$:
$$A = QB + R,$$
with $\deg R < \deg B$.
The following algorithm is based on the Newton iteration process presented for example in \cite{mca}, \S 9.1, which uses reciprocal polynomials. Our algorithm is an almost direct adaptation of it, the only subtlety here is that the map sending a skew polynomial to its reciprocal polynomial is not a morphism. 
\begin{lem}
For $n\geq 0$, we denote by $k[X,\sigma]_{\leq n}$ the subspace of skew polynomials of degree at most $n$. Let
$$ \begin{array}{ccccc}
\tau_n &:& k[X,\sigma]_{\leq n}& \rightarrow &k[X,\sigma^{-1}]_{\leq n}\\
&& \displaystyle \sum_{i=0}^{n} a_i X^i & \mapsto & \displaystyle \sum_{i=0}^n a_{n-i} X^i
\end{array}.
$$
Then $\tau_n$ is $k$-linear, bijective, and for all $P,Q \in k[X,\sigma]$, with $\deg P \leq n$ and $\deg Q \leq m$, we have:
$$ \tau_{n}(P) \tau_{m}(Q^{(n)}) = \tau_{m+n}(PQ).$$
\end{lem}
\begin{proof} The $k$-linearity is trivial, as well as bijectivity. Let $P = \sum_{i=0}^n a_i X^i$ and $Q = \sum_{j=0}^m b_jX^j$. Then the coefficient of $X^l$ in the product $PQ$ is
$$ c_{l} = \sum_{i+j = l} a_i\sigma^i(b_j).$$
Hence, the coefficient of $X^l$ in $\tau_{m+n}(PQ)$ is $c_{n+m -l} = \sum_{i+ j = l} a_{n-i} \sigma^{n-i}(b_{m-j})$. This is clearly the coefficient of $X^l$ in the product $\tau_n(P) \tau_m(Q^{(n)})$, computed in $k[X,\sigma^{-1}]$.
\end{proof}
Let us now describe the Euclidean division algorithm. Let $n = \deg A$ and $m = \deg B$. According to the previous formula, if $A = QB + R$ is the right-Euclidean division of $A$ by $B$, we have:
$$ \tau_{n}(A) = \tau_{n-m}(Q)\tau_m(B^{(n-m)}) + \tau_n(R).$$
Since $\deg R < m$, $\tau_n(R)$ is divisible by $X^{n-m+1}$. The idea is to compute an approximation of the left-inverse of $\tilde{B} = \tau_m(B^{(n-m)})$ in $k[\![ X, \sigma^{-1}]\!]$ (the ring of skew power series, which is defined in the obvious way, and is only used here to sketch the idea of the algorithm). Once we get such an approximation $\tilde{Q}$, truncated at precision $X^{n-m}$, we know that $\tau_n(A)\tilde{Q}\tilde{B} - \tau_n(A) \in  X^{n-m}k[ X, \sigma^{-1}]$, and by applying $\tau_{n}^{-1}$, we get the quotient $Q$.

Computing successive approximations of $\tilde {Q}$ is done by Newton iteration: let $B_0$ be the constant coefficient of $\tilde{B}$, we define $\tilde{Q}_0 = B_0^{-1}$, and $\tilde{Q}_{i+1} = 2\tilde{Q}_i -\tilde{ Q}_i \tilde{B}\tilde{Q}_i$, truncated at $X^{2^i}$.
\begin{lem}
For all $i \geq 0$, $\tilde{Q}_i\tilde{B} -1 \in X^{2^i}k[X,\sigma^{-1}]$.
\end{lem}
\begin{proof}
The proof goes by induction on $i$. By construction, $\tilde{Q}_0 \tilde{B} - 1 \in Xk[X,\sigma^{-1}]$. Now assuming that the result is true for some $i\geq 0$, we have:
$$\tilde{Q}_{i+1}\tilde{B} -1 = 2 \tilde{Q}_{i} \tilde{B} - \tilde{Q_i}\tilde{B}\tilde{Q}_i\tilde{B} -1 = -(1 - \tilde{Q}_i\tilde{B})^2 \in X^{2^{i+1}}k[X,\sigma^{-1}]$$
and we are done.
\end{proof}

\begin{algorithm}
\KwIn{$A,B \in k[X,\sigma]$ with $\deg A \geq \deg B$}
\KwOut{$Q,R \in k[X,\sigma]$ with $\deg R < \deg B$ such that $A = QB +R$}
$n=\deg A$;\, $m=\deg B$\;
$\tilde{B} = \tau_{m}(B^{(n)})$\;
$\tilde{Q} = \text{Coefficient}(\tilde B, 0)^{-1}$\;
$i=1$\;
\While{$i < n-m+1$}{
$\tilde{Q}= 2\tilde{Q} - \tilde{Q}(\tilde{B} \pmod {X^{i}})\tilde{Q} \pmod{X^{2i}$}\;
$i=2i$\;
}
$\tilde{Q} = (\tau_n(A) \pmod{X^{n-m}}) \tilde{Q} \pmod{X^{n-m}}$\;
$Q=\tau_{n-m}^{-1}(\tilde Q)$\;
$R = A - QB$\;
\Return{Q,R}\;
\caption{REuclideanDivision}
\end{algorithm}

\begin{prop}
The algorithm \verb"REuclideanDivision" returns the quotient and remainder of the right-division of $A$ of degree $n$ by $B$ of degree $m$ in $\tilde{O}(\SM(n,r))$ operations in $k^\sigma$.
\end{prop}
\begin{proof}
We have already seen that the result of this algorithm is correct. In order to compute $\tilde{B}$, $O(mr^2)$ operations are needed. The \verb"while" loop in the algorithm has $\log_2(n-m+1)$ steps, and at the $i$-th step, we compute the product of skew polynomials of degree $2^i$, so the global complexity of this is $\sum_{i=0}^{\log_2(n-m+1)} \SM(2^i, r) = \tilde{O}(\SM(n-m,r))$. Computing $(\tau_n(A) \pmod{X^{n-m}}) \tilde{Q}$ has the same complexity. Finally, we compute the product $QB$ in $\SM(\max{m, n-m},r)$ operations, and $R = A- QB$ in $\tilde{O}(\SM(n,r))$ operations.
\end{proof}
\subsubsection{Greatest common divisors and lowest common multiples} 
This section describes an algorithm adapted directly from Algorithm 11.4 of \cite{mca}, to compute the right-gcd $R$ of two skew polynomials $A$ and $B$, together with skew polynomials $U,V$ such that $UA + VB = R$. As we have seen before, this also gives almost directly the left-lcm of $A$ and $B$. 
%We shall state without proof two lemmas that are in \cite{mca} and whose proof is essentially the same as for commutative polynomials, and then give a description of the algorithm and compute its complexity.

This algorithm relies on the fact that in the Euclidean division, the highest-degree terms of the quotient only depend on the highest-degree terms of the dividend and divisor. If $A \in k[X,\sigma]$ and $n \in \N$, with $A = \sum_{i=0}^{d}a_i X^i$ of degree $d$, we set $A_{(n)} = \sum_{i=0}^n a_{d-i}X^{n-i}$, with the convention that $a_j = 0$ for $j \notin \{0, \cdots, d\}$. Then, for $n\geq 0$, $A_{(n)}$ is a skew polynomial of degree $n$, and for $n<0$. Note that for all $i \geq 0$, $(AX^i)_{(n)} = A_{(n)}$.

\begin{defi}
If $A,B, A^*, B^* \in k[X,\sigma]$ with $\deg A \geq \deg B$ and $\deg A^* \geq \deg B^*$, and $n\in \Z$, we say that $(A,B)$ and $(A^*,B^*)$ \emph{coincide up to $n$} if
\begin{enumerate}
\item $A_{(n)} = A^*_{(n)}$,
\item $B_{(n -(\deg P - \deg Q))} = B^*_{(n - (\deg P^* - \deg Q^* ))}$
\end{enumerate}
\end{defi}
Then we have the following:
\begin{lem}[\cite{mca}, Lemma 11.1.]
Let $n \in \Z$, $(A,B)$ and $(A^*,B^*) \in (k[X,\sigma]\setminus \{0\})^2$ that coincide up to $2n$, with $n \geq \deg A - \deg B \geq 0$. Define $Q, R, Q^*, R^*$ as the quotient and remainder in the right-divisions:
$$\begin{array}{cccl}
A &=& QB + R,&\text{with }\deg R < \deg B,\\
A^* & =& Q^*B^* + R^*&\text{with }\deg R^* < \deg B^*.
\end{array}$$
Then $Q = Q^*$, and either $(B,R)$ and $(B,R^*)$ coincide up to $2(n - \deg Q)$ or $R = 0$ or $n - \deg Q < \deg B - \deg R$.
\end{lem}
Now, we want to carry this approximation further down the sequence of quotients when doing the Euclidean algorithm. For $A_0, A_1,A_0^*,A_1^* \in k[X,\sigma]$ monic, with $\deg A_0 > \deg A_1$ and $\deg A_0^* > \deg A_1^*$, we write:
$$\begin{array}{rclrcl}
A_0 &= &Q_1 A_1 + \rho_2 A_2,& A_0^* &= &Q_1^* A_1^* + \rho_2^* A_2^*,\\
&\vdots & & &\vdots &\\
A_{i-1} &= &Q_i A_i + \rho_{i+1} A_{i+1},& A_{i-1}^* &= &Q_i^* A_i^* + \rho_{i+1}^* A_{i+1}^*,\\
&\vdots & & &\vdots &\\
A_{\ell -1}& = & Q_\ell A_\ell, & A^*_{\ell^* -1}& = & Q^*_{\ell^*} A^*_{\ell^*},
\end{array}
$$
with for all $i$, $\deg A_{i+1} < \deg A_i$, with $\rho_i \in k^\times$ and $A_i$ monic.
From this sequence, we define for $1\leq i \leq \ell$, $m_i = \deg Q_i$, $n_i = \deg A_i$, and for $n \in \N$,
$$ \eta(n) = \max\left\{0 \leq j \leq \ell~|~\sum_{1 \leq i \leq j} m_i \leq n\right\}.$$
We define analogously $m_i^*, n_i^*$ and $\eta^*$. Then the following lemma quantifies how much the first results in the Euclidean algorithm only depend on the highest-power terms of the entires;
\begin{lem}[\cite{mca}, Lemma 11.3.]
Let $n \in \N$, $h = \eta(n)$ and $h^* = \eta^*(n)$. If $(A_0, A_1)$ and $(A_0^*, A_1^*)$ coincide up to $2n$, then $h = h^*$, $Q_i = Q_i^*$ and $\rho_{i+1} = \rho_{i+1}^*$ for $1 \leq i \leq h$.
\end{lem}
Let us now describe the extended Euclidean Algorithm.

\begin{algorithm}[H]
\KwIn{$A_0, A_1 \in k[X,\sigma]$ monic, $n_0 = \deg A_0 \geq \deg A_1 = n_1$ and $n \in \N$ with $0 \leq n \leq n_0$.}
\KwOut{$M \in \mathcal{M}_2(k[X,\sigma])$ such that $M \begin{pmatrix} A_0\\ A_1\end{pmatrix} = \begin{pmatrix} A_{h}\\ A_{h+1}\end{pmatrix}$ with $h = \eta(n)$.}
\lIf{$A_1 = 0$ or $n < n_0 - n_1$}{\Return $0$, $\begin{pmatrix}1&0\\ 0&1 \end{pmatrix}$\;}
$d = \lfloor n/2\rfloor$\;
$R =$ FastExtendedRGCD$({R_0}_{(2d)}, {R_1}_{(2d - (n_0 - n_1))}, 2d, 2d-(n_0 - n_1), d)$\;
$\begin{pmatrix} A'_0 \\ A'_1 \end{pmatrix} = R\begin{pmatrix}A_0\\ A_1 \end{pmatrix}$;\, $\begin{pmatrix} n'_0\\ n'_1 \end{pmatrix} = \begin{pmatrix} \deg A'_0 \\ \deg A'_1\end{pmatrix}$\;
\lIf{$A'_1 = 0$ or $n < n_0 - n_j$}{\Return $R$\;}
$Q_j =A'_0/A'_1$;\, $\rho'_2 = \text{LeadingCoefficient}(A'_0 \mod A'_1)$\;
$A'_2 = (\rho'_2)^{-1} (A'_0 \mod A'_1)$;\, $n'_2 = \deg A'_2$\;
$d^* = n - (n_0 - n'_1)$\;
$S =$ FastExtendedRGCD$(A'_1,A'_2, 2d^*, 2d^* - (n'_1 - n'_2), d^*)$\;
$M_j = \begin{pmatrix} 0 & 1 \\ (\rho'_2)^{-1} & (\rho'_2)^{-1} Q_j\end{pmatrix}$\;
\Return{$S \cdot M_j \cdot R$\;}
\caption{FastExtendedRGCD}
\end{algorithm}
When executed for $n = n_0$, the above algorithm gives an immediate way to compute the right-gcd and left-lcm of $A_0$ and $A_1$. Indeed, in this case, we get a matrix $M = \begin{pmatrix} U_0 &U_1\\V_0 & V_1 \end{pmatrix}$ such that $U_0 A_0 + U_1 A_1  = \text{rgcd}(A_0, A_1)$, and $V_0 A_0  = -V_1 A_1 = \text{llcm}(A_0,A_1)$.

\begin{thm}[\cite{mca}, Theorem 11.5.]\label{gcd-lcm}
The algorithm \verb"FastExtendedRGCD" works correctly and uses at most $O(\SM(n,r)\log n)$ operations in $k^\sigma$ if $A_0$ has degree $d \leq 2n$. In particular, it allows to compute the rgcd and llcm with $\tilde{O}(\SM(d,r))$ operations in $k^\sigma$.
\end{thm}
\begin{proof}
The proof for correctness is exactly the same as the one in \cite{mca} and relies on the previous two lemmas. Let us give more details about the complexity of the algorithm. Denote by $T(n_0,n_1,n)$ the time needed to call  \verb"FastExtendedRGCD" on two skew polynomials $A_0, A_1$ of degrees $n_0, n_1$, with parameter $n$. Set $d = \lfloor n_0/2\rfloor$. Then we have:
$$ T(n_0,n_1) \leq T(2d,2d - (n_0 - n_1),d) + T(2d^*,2d^* -(n_j - n_{j-1}),d^*) + O(\SM(d,r)).$$
The term $\SM(n_0,r)$ here comes from the multiplications needed from matrix multiplications (all the polynomials in these matrices have degree at most $n_0$) and one due to the Euclidean division algorithm. The result follows by induction from the fact that $d^* = \lceil n/2 \rceil$.
\end{proof}

\subsection{Computing the norm}
In this section, we give algorithms to compute the reduced norm of a skew polynomial. Let $N = N_{k/k^\sigma}$ be the norm from $k$ to $k^\sigma$. Let $P \in k[X,\sigma]$ of degree $d$. We give two different ways to compute the norm, depending on whether $d$ is greater or smaller than $r$. Let us start with the first case, $d < r$:
\begin{prop}[\cite{jac}, Proposition 1.7.1]
\label{prop:jac}
Let $P \in k[X,\sigma]$ of degree $d<r$, $P = \sum_{i=0}^d a_i X^i$. Then
$$ \n(P) = (-1)^{rd}N(a_0) + (-1)^{r(d-1)} N(a_1)X^{r} + \cdots + N(a_d) X^{rd}.$$
\end{prop}
This Proposition gives a direct way to compute $\n(P)$: this is done by computing the norms of its $d$ coefficients. Since all the conjugates of an element of $k$ can be computed in $\tilde O(r^2)$ operations in $k^\sigma$, and the product of $r$ elements of $k$ requires $\tilde{O}(r^2)$ operations in $k^\sigma$, the norm of an element of $k$ can be computed in $\tilde{O}(r^2)$ operations in $k^\sigma$. Hence, by Proposition \ref{prop:jac}, we get an algorithm to compute $\n(P)$ in $\tilde{O}(dr^2)$ operations in $k^\sigma$ when $r < d$.

Let us now address the case $d \geq r$. We use the fact that $\n(P)$ is the determinant of multiplication by $X^r$ on $D_P$, seen as a $k[X^r]$-module. Let $t \in k$ be a primitive element over $k^\sigma$, and let $\pi_t \in k^\sigma[X^r]$ be its minimal polynomial over $k^\sigma$. Let $R_0 \in k^\sigma[X^r]$ be a polynomial of degree $n > d/r$. Let $R$ be the polynomial obtained by composition: $R = \pi_t\circ R_0$. We work in the ring $\mathcal{A} = k^\sigma[X^r]/R$. 

The idea is the following: if $R$ is irreducible, then $\mathcal{A}$ is a field extension of $k^\sigma$, and there is a natural embedding of $k$ into $\mathcal{A}$, mapping $t$ to $R_0$. Then we can write the matrix of multiplication by $P$ in $k[X,\sigma]$ seen as a module over $k[X^r]$, and map it to a matrix with coefficients in $\mathcal{A}$. Then we can compute the determinant of this matrix, which is the image $\nu$ of the norm of $P$ by the map $k[X^r] \rightarrow \mathcal{A}$. Since it is known to be a polynomial with coefficients in $k^\sigma$ of degree $d$, and since $[\mathcal{A}:k^\sigma] > d$, the coefficients of the $\n(P)$ are exactly the coefficients of $\nu$ written in the canonical basis of $\mathcal{A}$.

Actually, all of the above still holds if $\mathcal{A}$ is not a field, except that we may not use algorithms for determinants over fields to compute $\nu$. However, we can still obtain this determinant efficiently by computing the Hermite normal form of the matrix of multiplication by $P$ in the Euclidean domain $\mathcal{A}$. So in practice, all we have to do is write the matrix of multiplication by $P$ as a matrix with coefficients in $k[X^r]$. Write $P = P_0 + P_1 X + \cdots + P_{r-1}X^{r-1}$. As stated in the proof of Lemma \ref{norm-otherdef}, in the canonical basis $1, X, \ldots, X^{r-1}$, the matrix of multiplication by $P$ is:

$$\begin{pmatrix}
P_0 & X^r\sigma(P_{r-1}) & \ldots & \ldots & X^r \sigma^{r-1}(P_1)\\
P_1 & \sigma(P_0) & \ddots & \ddots & \vdots \\
\vdots & \ddots & \ddots & \ddots & \vdots \\
\vdots & \ddots & \ddots & \ddots & X^r\sigma^{r-1}(P_{r-1})\\
P_{r-1} & \cdots & \cdots & \cdots & \sigma^{r-1}(P_0)
\end{pmatrix}.$$
We map this matrix to $\mathcal{A}$ by taking $X^r$ to its residue class modulo $R$, and $t$ to $R_0 \pmod R$. Then, we compute its determinant $\nu$ (using Smith normal form), and we can read the coefficients of $\n(P)$ on $\nu$.

If $P$ has degree $d$, the complexity of these operations is: $O(dr^2)$ operations to compute all the conjugates of the $P_i$'s under the action of the Frobenius. Multiplication by $X^r$ is free in $k[X^r]$. This yields a total of $O(d r^2)$ operations to compute the matrix, and then $O(dr^3)$ operations in $k^\sigma$ to get its determinant. Hence, $P$ can be computed in $O(dr^3)$ operations in $k^\sigma$.

To sum it up, if $r< d$, then we can compute $\n(P)$ in $O(dr^2)$, and if $r\geq d$, we can compute it in $O(dr^3)$ operations in $k^\sigma$.

\subsection{A fast factorization algorithm}

Let $P \in k[X,\sigma]$ be a monic polynomial. Our aim is to give an algorithm to compute a factorization of $P$ as a product of irreducible skew polynomials. The idea of the algorithm is to reduce that problem to the problem of factoring polynomials of type $(e)$ (using $\text{rgcd}$'s with factors of the norm of $P$) and then to factor polynomials of type $(e)$. For the sake of brevity, in the algorithms we will use the notation $A/B$ for the quotient of the right-division of $A$ by $B$.
\paragraph{Reduction to the type-$(e)$ case \\}
The following algorithm recursively computes the $\text{rgcd}$ of a polynomial $P$ with a central polynomial (whose irreducible factors are all irreducible factors of $\n(P)$) and writes it as a product of polynomials of type $(e)$ (for some integer $e$ depending on the factor).

\begin{algorithm}
\KwIn{$P \in k[X,\sigma]$, $(N_1, \ldots, N_m)$ irreducible such that $\n(P) = \prod N_i$, ordered by nondecreasing degree}
\KwOut{$P_{1,1}, P_{1,2}, \ldots, P_{1,m_1}, \ldots, P_{n,1}, \ldots, P_{n, m_n} \in k[X,\sigma]$ and $N_1, \ldots, N_n \in k^\sigma[X^r]$ irreducible such that $P = \prod_i \prod_j P_{i,j}$ and each $P_{i,j}$ has type $e_j$ and norm $N^{e_j}$}
$d_1 = \deg N_1$\;
\lFor{$1 \leq i \leq m-1$}
{$d_{i+1} = d_i + \deg N_{i+1}$\;}
$d = d_m$; \, $\delta = d/\log d$\;
$i = \min \{1 \leq j \leq m - 1 ~|~ d_j > d + \delta/2 \}$\;
\eIf{$[d - \delta/2, d + \delta/2] \cap \{ d_1, \ldots, d_{m-1}\} = \emptyset$}
{$j = m$\;
\While{ $j\geq i$}
{$P_j = \text{rgcd}(P,N_j)$\;
$j = j - \deg P_j/\deg N_j$\;
$P = P/P_j$\;
}
\Return{Type\_e\_Factorization$(P, (N_1, \ldots, N_{i-1}))$, $\{P_j ~|~ i \leq j \leq m\}$\; }
}{
$M = N_i \cdots N_m$\;
$Q_1 = \text{rgcd}(P,M)$;\,
$Q_2 = P/M$\;
\Return{Type\_e\_Factorization$(Q_2, (N_1, \ldots, N_i))$, Type\_e\_Factorization$(Q_1, (N_i, \ldots, N_m))$}\;
}
\caption{Type\_e\_Factorization}
\end{algorithm}

\paragraph{Factoring a polynomial of type $(e)$\\}
Let us now explain how to factor a polynomial $P$ of type $(e)$. Clearly, $\n(P) = N^e$ with $N \in k^\sigma[X^r]$ irreducible. In this case, we know that $P$ is a divisor of $N$, we write $PQ = N$ and will work with $k[X,\sigma]Q/(N)$ rather than $k[X,\sigma]/k[X,\sigma]P$. Let $R \in k[X,\sigma]Q/(N)$. Right-multiplication by $R$ is an endomorphism of $k[X,\sigma]Q/(N)$ that is a $E$-vector space of dimension $e$. Hence, there exist $\lambda_0, \ldots, \lambda_{e-1} \in E$ such that $R^e = \sum_{i=0}^{e-1} \lambda_i R^i$. Now assume that $F(T) = T^e - \sum_{i=0}^{e-1} \lambda_iT^i \in E[T]$ has a root $\alpha \in E$. Then $R - \alpha$ is a zero-divisor in $k[X,\sigma]Q/(N)$. Indeed, $\alpha$ is an eigenvalue of multiplication by $R$, so there exists some $S \in k[X,\sigma]Q/(N)$ such that $S(R-\alpha)$ is zero in $k[X,\sigma]Q/(N)$. Write $R = \tilde{R}Q$ and $S = \tilde{S}Q$, then $\tilde{S}(Q\tilde{R} - \alpha)$ is divisible by $P$, so that the right gcd of $Q\tilde R - \alpha$ is a divisor of $P$. Moreover, if $\alpha$ is the only eigenvalue of the multiplication by $R$ (with multiplicity one), then this divisor is irreducible. We will see that this happens with good probability.\\
Once we get an irreducible factor, we can proceed recursively to factor $P$. However, we can also use a slightly more efficient trick relying on the knowledge of an irreducible factor. Assume we know an irreducible right factor $P_1$ of $P$, and write $P = P_2 P_1$. Let $R \in k[X,\sigma]$ and let $A = \text{rgcd}(P, P_1QR)$. Now let $B = \text{llcm}(A,P_1) = \tilde{B}P_1$. Since $P$ is a right multiple of both $P_1$ and $A$, $B$ is a divisor of $P$. Hence, $\tilde{B}$ is a divisor of $P_2$. In general, $A$ and $\tilde{B}$  should have the same degree as $P_1$, yielding an irreducible factor of $P_2$. The precise probability study will appear in \S \ref{complexity}.
%The idea to find a nontrivial divisor of $P$ is to find $R \in k[X,\sigma]/k[X,\sigma]P$ such that right-multiplication by $R$ is not injective. If $R$ is such a polynomial, then $P$ and $R$ must have a nontrivial common factor. Our strategy is to randomly pick some $R \in k[X,\sigma]/k[X,\sigma]P$, to compute the characteristic polynomial of the multiplication by $RQ$ acting on $k[X,\sigma]Q/k[X,\sigma]N$, and then to find a root $\lambda \in E$ of this polynomial. In this case, multiplication by $RQ - \lambda$ is not injective in $k[X,\sigma]Q / k[X,\sigma]N$, so the same holds for multiplication by $R -\lambda$ in $D_P$. Then $R - \lambda$ and $P$ must have a common factor. Moreover, if $\lambda$ is the only eigenvalue of the endomorphism of multiplication by $RQ$ in $E$ (and $\lambda$ is simple), then that common factor is irreducible. We will see that this happens with good probability.
The following two algorithms describe how to factor a polynomial $P$ of type $(e)$: the first one finds one irreducible factor of $P$, and the second one performs the "lcm trick" to factor $P$ as a product of irreducibles given one irreducible right factor.

\begin{algorithm}[H]
\KwIn{$(P,N) \in k[X,\sigma]\times k^\sigma[X ^r]$ such that of type $(e)$, $\n(P) = N^e$ and $N$ is irreducible}
\KwOut{An irreducible right-divisor of $P$}
$E = k^\sigma[X^r]/(N)$\;
$Q = N/P$\;
\While{true}{
  $\tilde{R} = \text{RandomElement}(k[X,\sigma]/k[X,\sigma]P)$\;
  $R_0 = \tilde RQ$\;
  \lFor{$0 \leq i \leq e-1$}{$R_{i+1} = R_0 R_i$\;}
  Find $\lambda_0,\ldots, \lambda_{e-1} \in E$ such that
  $R_e = \sum_{i=0}^{e-1}{\lambda_i R_i}$\;
  $F(T) = T^e - \sum_{i=0}^{e-1} \lambda_i T^i$\;
  \If{$F$ has a simple root $\alpha$ in $E$}
    {$P_1 = \text{rgcd}(P, Q\tilde{R} - \alpha)$\;
    \Return{$P1$}\;}
}
\caption{FirstFactor}
\end{algorithm}

\begin{algorithm}[H]
\KwIn{$(P,N,P_1) \in k[X,\sigma]\times k^\sigma[X ^r]$ such that of type $(e)$, $\n(P) = N^e$, $N$ is irreducible and $P_1$ is an irreducible right factor of $P$}
\KwOut{Irreducible polynomials $P_1, \ldots, P_e$ such that $P = P_e\cdots P_1$}
$Q = N/P$\;
\For {$1\leq i \leq e-1$}{
  \While{true}{
    $R = \text{RandomElement}(k[X,\sigma]/k[X,\sigma]P)$\;
    $A = \text{rgcd}(P,P_1QR)$\;
    $\tilde{B} = \text{llcm}(P_1,A)/P_1$\;
    \If{$\deg B = \deg P_1$}{
      $P_i = \tilde{B}$\;
      \textbf{break}\;
    }
  }
}
\Return {$P_1, \ldots, P_e$\;}
\caption{FactorStep}
\end{algorithm}

Glueing together the three previous algorithms, we get a complete 
factorization algorithm. We assume that the function Factorization 
returns the factorization of a (commutative) polynomial as a product of 
irreducible polynomials ordered by their degrees.

\begin{algorithm}[H]
\KwIn{$P \in k[X,\sigma]$}
\KwOut{A list of irreducible polynomials $(P_1, \ldots, P_m)$ such that $P = P_m\cdots P_1$}
$N = \n(P)$\;
$N_1 \cdots N_m =$ Factorization$(N)$\;
$(G_{1,1}, \ldots, G_{n,m_n}) = \text{Type\_e\_Factorization}(P, (N_1, \ldots, N_m))$\;
\For{$1 \leq i \leq m$}{
\For{$1 \leq j \leq m_i$}{$P_{i,j,1}, \ldots, P_{i,j,e_{ij}} = \text{FactorizationStep}(G_{i,j}, G_i)$\;}}
\Return{$(P_{i,j,l})$\;}
\caption{SkewFactorization}
\end{algorithm}

\subsection{Complexity}
In this section, we analyze the complexity of the factorization algorithm. The complexity will be expressed in terms of the degree $d$ of the skew polynomial that is to be factored, the degree $r$ of $k/k^\sigma$, and the cardinal $q$ of $k^\sigma$.
\subsubsection{Complexity of the steps}
Let us detail the complexity of the steps of our factorization algorithm.

\paragraph{Type-$(e)$-factorization}
We have the following lemma, giving the complexity of the algorithm \verb"Type_e_Factorization".
\begin{lem}
Let $P \in k[X,\sigma]$ and let $N_1, \ldots, N_m \in k^\sigma[X^r]$ be irreducible polynomials such that $P$ divides $N_1\cdots N_m$ in $k[X,\sigma]$. Then the algorithm \verb"Type_e_Factorization" applied to $P$ and $N_1, \ldots, N_m$ returns a correct result with $\tilde{O}(dr^3)$ operations in $k^\sigma$.
\end{lem}
\begin{proof}
Let us prove the result by induction on $d$. Let $(N_1,a_1), \ldots, (N_m,a_m)$ be the irreducible polynomials that are given as arguments, and $\delta_i = \deg N_i$ for $1 \leq i \leq m$. We assument that the $N_i$'s are ordered so that the sequence of $\delta_i$ is nondecreasing. There are two cases to look at.

If there exists $1 \leq i \leq m$ and $1 \leq a \leq a_i$ such that
$$\sum_{j = 1}^{i-1} a_j \delta_j + a\delta_i \in \left[\frac{d}{2}\left(1 - \frac{1}{\log d}\right), \frac{d}{2}\left(1 + \frac{1}{\log d}\right)\right],$$
then we choose the minimal $(i,a)$ (for the lexicographical order) having this property. We write $N_l = N_i^a\prod_{j=1}^{i-1}N_j^{a_j}$, and $N_r = N/N_l$. Then we write $P_r = \text{rgcd}(P,N_r)$, and define $P_l$ as the quotient in the right-division of $P$ by $P_r$. The algorithm is then applied to $(P_l, N_l , (N_1,a_1), \ldots, (N_i, a))$ and $(P_r, N_r, (N_i, a_i - a), \ldots, (N_m, a_m))$.

The number of operations needed for this is denoted by $C(d,r)$. In this case, we have:
$$ C(d, r) \leq 2C\left( d\left (1 + \frac{1}{\log d}\right),r\right) + \tilde{O}(\SM(dr,r)).$$
Indeed, the operations we have to do before starting the recursive steps are: computing a product of (commutative) polynomials in $k^\sigma[X^r]$ such that the sum of their degrees is less than $d\left (1 + \frac{1}{\log d}\right)$, computing the right gcd of $P$ with a polynomial of degree less than $dr$, and dividing $P$ by this gcd. The most expensive part is the computation of the gcd, and it costs $\tilde{O}(\SM(dr,r))$.

In the other case, there is no $(i,a)$ such that
$$\sum_{j = 1}^{i-1} a_j \delta_j + a\delta_i \in \left[\frac{d}{2}\left(1 - \frac{1}{\log d}\right), \frac{d}{2}\left(1 + \frac{1}{\log d}\right)\right].$$
Hence, for $(i,a)$ such that $\sum_{j = 1}^{i-1} a_j \delta_j + a\delta_i > \frac{d}{2}\left(1 + \frac{1}{\log d}\right)$, we know that $\delta_i > \frac{d}{\log d}$, and there are at most $\log d$ such couples $(i,a)$. In this case, the algorithm is to compute $N_l$, $N_r$ as before, and then the successive gcd's of $P$ the $N_i$'s having the previous property, and apply the algorithm with the last quotient $P_l$ and $N_l$.

There are at most $\log d$ rgcd's of a skew polynomial of degree at most $d$ with skew polynomials of degree at most $dr$, which takes $\tilde{O}(\SM(dr,r))$ operations, and all the other computations are cheaper than this. Again, we have:
\begin{eqnarray*}
C(d, r) &\leq &C\left( d\left (1 + \frac{1}{\log d}\right),r\right) + \tilde{O}(\SM(dr,r)) \\
&\leq &2\cdot C\left( d\left (1 + \frac{1}{\log d}\right),r\right) + \tilde{O}(\SM(dr,r)).
\end{eqnarray*}
Let us assume that the $\tilde{O}(\SM(dr,r))$ appearing in the above inequality is $\leq cdr^3\log^\alpha d$ for some constants $c, \alpha$ (we use the fact that $\SM(d,r) = \tilde{O}(dr^2)$). We are going to show that there exists a constant $c'$ such that
$$C(d,r) \leq c' dr^3\log^{\alpha + 1}d.$$
We want to have:
$$C(d,r) \leq 2 c' \frac{d}{2}\left(1 + \frac{1}{\log d} \right) r^3 \log^{\alpha+1}\left(\frac{d}{2}\left(1 + \frac{1}{\log d} \right) \right) + cdr^3 \log^\alpha d.$$
This implies that:
$$C(d,r) \leq c'dr^3\log^{\alpha+1}d\left( 1 - \frac{\log 2}{\log d} + O\left( \frac{1}{\log^2 d}\right)\right)^{\alpha+1} + c'dr^3\log^{\alpha} d + c dr^3 \log^{\alpha +1}d.  $$
If we choose $c'$ such that $c' + c - c'(\alpha +1)\log 2 + O\left( \frac{1}{\log^2 d}\right) \leq 0$ for $d$ large enough, then induction shows that for $d$ large enough,
$$ C(d,r) \leq c' dr^3 \log^{\alpha + 1} d.$$
Since it is possible to choose such a $c'$, the proof is complete.
\end{proof}

\paragraph{FirstFactor}
We shall detail the complexity of all the steps of this algorithm. In the following, $P$ has type $(e)$ and norm $N^e$, with $e \leq r$. The degree of $N$ as an element of $k^\sigma[X^r]$ is $\delta$, so that the degree of $P$ is $\delta e$.
\begin{enumerate}
\item Compute $Q \in k[X,\sigma]$ such that $PQ = N$. This Euclidean division can be done with complexity $\SM(dr,r)$. Note that this step is done only once even if the loop fails to find a divisor.
\item Choose a random element $R \in k[X,\sigma]/k[X,\sigma]P$ and compute $RQ, \ldots, (RQ)^{e}$ modulo $N$. This requires $e$ multiplications of skew polynomials of degree $\delta r$ plus one reduction modulo $N$ at each step. After having remarked the reduction modulo $N$ of a skew polynomial is equal to its reduction modulo $N$ in the ring of usual polynomials, we see that it costs only $\tilde O(\delta^2 r)$ operations in $k^\sigma$. The whole cost of this step is then $O(e \cdot \SM(\delta r,r))$.
\item Find a linear dependence between the powers of $RQ$ of the following
for:
\begin{equation}
\label{eq:FT}
\sum_{i=0}^e a_i (RQ)^i = 0.
\end{equation}
where all $a_i$'s are in $E$.
Even though the element $(RQ)^i$ naturally live in a space of dimension 
$r^2$ over $E$, we know the first $e$ of them are linearly dependent, 
and we can work in a vector space of dimension $e$ over $E$ by 
projection. Hence, the complexity of this step is $\delta \cdot \MM(e)$. 
\item Check whether the polynomial $F(T) = \sum_{i=0}^e a_i T^i$ (where 
the $a_i$'s are defined by formula \eqref{eq:FT}) has a root in $E$. For 
this, it is enough to compute the gcd of $F$ with $T^{\# E} - T$. Noting 
that $\# E = q^\delta$ with $q = \# k^\sigma$, we can first compute 
$T^{\# E}$ modulo $F(T)$ by first raising $T$ to the $q$-th power modulo 
$F(T)$ (using classical fast exponentiation) and then performing $O(\log 
\delta)$ modular compositions. Using Corollary 5.2 of \cite{ku}, this can 
be done in 
$\tilde O(\delta \log^2 q + e^{1+\varepsilon} (\delta \log q)^{1+o(1)})$
bit operations, for all $\varepsilon > 0$ (the first term corresponds to 
the fast exponentiation and the second to the modular compositions). 
It then remains to compute 
the gcd of two polynomials over $E$ of degree $\leq e$, which can be 
achieved with $\tilde O(e \delta)$ more operations in $k^\sigma$.
\item Compute the right gcd of $P$ with a skew polynomial of degree 
$\delta r$, which costs $\tilde{O}(\SM(\delta r,r))$ operations in 
$k^\sigma$. Note that this step is done only once even if the loop fails 
to find a divisor.
\end{enumerate}

Since any operation in $k^\sigma$ requires $\tilde O(\log q)$ bit
operations, the total complexity of this algorithm is 
$$\tilde{O}(\SM(\delta r,r) \cdot e \log q + \MM(e) \cdot \delta \log q
+ \delta \log^2 q + e^{1 + \varepsilon} (\delta \log q)^{1 + o(1)})$$
bit operations. Using $\SM(n,r) = \tilde O(nr^2)$ and $\MM(n) = O(n^3)$
and noting that $e \leq r$, this becomes
$$\tilde{O}(\delta e r^3 \log q + \delta \log^2 q + e^{1 + \varepsilon} 
(\delta \log q)^{1 + o(1)})$$
for all $\varepsilon > 0$.
We will see in \S \ref{proba} why the probability of failure is bounded from below independently on the data of the problem.

\paragraph{FactorStep}
This algorithm computes a factorization of $P$ (still of type $(e)$) when $P$, knowing a factor of $P$. The next irreducible factor is computed with one rgcd and one llcm between polynomials of degree at most $\delta r$. This operation may fail (in which case, we repeat it with a new random input) but we will see in the next section that the propability of failure is very small. Hence, in order to compute the complexity, it is safe to assume that failures never append. As it is shown in \S \ref{gcd-lcm}, this has complexity $\tilde{O}(\SM(\delta r,r))$. Hence the complexity of this step is $\tilde{O}(\delta r^3)$.

\subsubsection{Global complexity}
Let us sum up all the previous step complexities to give the complexity of the whole factorization algorithm.
\begin{thm}\label{complexity}
The algorithm \verb"SkewFactorization" runs in
$$\tilde{O}(dr^3 \log q + d \log^2 q + d^{1+\varepsilon} (\log q)^{1 + o(1)} + F(d,k^\sigma))$$
bit operations to factor a skew polynomial of degree $d$.
Here, $F(d,K)$ denotes the complexity of the factorization of a 
(commutative) polynomial of degree $d$ over the finite field $K$.
\end{thm}
\begin{proof}
Computing the norm of $P \in k[X,\sigma]$ of degree $d$ takes $O(dr^3)$ 
operations in $k^\sigma$. Factoring the norm $\n(P)$ (that has degree 
$d$ as an element of $k^\sigma[X^r]$) takes by definition 
$F(d,k^\sigma)$ operations in $k^\sigma$. Then, the 
algorithm \verb"Type_e_Factorization" runs in $\tilde{O}(dr^3)$ 
operations in $k^\sigma$. Let $P_1, \ldots, P_m$ be the factors of $P$ 
obtained after \verb"Type_e_Factorization". Assume that $P_i$ has type 
$e_i$ and degree $\delta_ie_i$. Then for each $i$, the factorization of 
$P_i$ takes $\tilde{O}(\delta_i e_i r^3 \log q + \delta_i e_i \log^2 q
+ e^{1 + \varepsilon} (\delta \log q)^{1 + o(1)})$ bit operations (it uses 
\verb"FirstFactor" and \verb"FactorStep"). So, to factor $P$ given its 
``type-(e)-factorization'', we need $\tilde{O}(dr^3 \log q + d \log^2 q
+ d^{1 + \varepsilon} (\log q)^{1+o(1)})$ bit
operations. Putting all the steps together, we get the desired complexity.
\end{proof}

\begin{rmq}
Of course, this result is true provided that the probabilities of success of the probabilistic parts of the algorithm are bounded from below independently of $d$ and $r$. This is what we will show in the next part.
\end{rmq}

\subsubsection{Probability of finding a factor}\label{proba}
The function \texttt{FactorStep} finds an irreducible factor of $P$ whenever the random endomorphism  "multiplication by $R$" has exacly one (simple) eigenvalue in $E$. 
By Corollary \ref{random}, if $R$ is uniformly distributed in $k[X,\sigma]Q/(N)$, then $m_R$ is uniformly distributed in $\End(k[X,\sigma]Q/(N))$. Therefore, we want to evaluate the probability for an endomorphism of a $E$-vector space of dimension $e$ to have a unique eigenvalue in $E$.

Let $B_d$ be the probability that a $d \times d$ matrix with coefficients in $E$ has $0$ as a simple, unique eigenvalue in $E$. Obviously, setting $q = \#E$, this is $\frac 1 q$ times the probability that a $d\times d$ matrix has a simple, unique eigenvalue in $E$. We can write:
$$ q^{d^2}B_d = \#\mathbf{P}(E^d)\cdot q^{d-1}q^{(d-1)^2}\cdot A_{d-1} = \frac{1 - \frac{1}{q^d}}{1-\frac{1}{q}} \cdot \frac{A_{d-1}}{q}. $$
where $A_i$ denotes the probability that a $i \times i$ matrix with coefficients in $E$ has \emph{no} eigenvalue in $E$.\\
Let us now detail how to obtain a bound on $A_i$. By \cite{pra}, Theorems 4.1 and 4.2, we get the formula for the generating series:
$$ \sum_{i=0}^{+\infty} A_i z^i = \frac{1}{1-z}G(z)$$
where $G(z) = \prod_{i=1}^{+ \infty}\left( 1 - \frac{z}{q^i}\right)^{q-1}$. If we write $G(z) = \sum_{i} C_i z^i$, then for all $i \geq 0$, $A_i = \sum_{j=0}^i (-1)^j C_j$. 
\begin{lem}
We have the following formulas:
\begin{itemize}
\item $A_0 = C_0 = 1$
\item $C_1 = 1$
\item $C_2 = \frac{q}{2(q+1)}$
\end{itemize}
\end{lem}
\begin{proof}
The first two assertions follow easily from identifying the coefficients of $1$ and $z$ in the power series $G$. For the third formula, identifiying the coefficient of $z^2$ gives:
$$ C_2 = \sum_{i=1}^{+\infty} \frac{(q-1)(q-2)}{2q^{2i}} + \sum_{i < j} \frac{(q-1)^2}{q^{i+j}}.$$
The result then follows from the usual formulas for sums of geometric progressions.
\end{proof}
Next, remark that:
$$\sum_{i=0}^{+\infty} C_i = \prod_{i \geq 1}\left( 1 + \frac{1}{q^i}\right)^{q-1} \quad \text{and} \quad \sum_{i=0}^{+\infty} (-1)^iC_i = \prod_{i \geq 1}\left( 1 - \frac{1}{q^i}\right)^{q-1}.$$
Combining both expressions, we get:
$$2 \cdot \sum_{i=0}^{+\infty} C_{2i+1} = \prod_{i \geq 1}\left( 1 + \frac{1}{q^i}\right)^{q-1} -  \prod_{i \geq 1}\left( 1 - \frac{1}{q^i}\right)^{q-1}.$$

Studying the function $q \mapsto \prod_{i \geq 1}( 1 + q^{-i})^{q-1} - 
\prod_{i \geq 1}( 1 - q^{-i})^{q-1}$ are non-decreasing, we find that the
the sum $\sum_{i=0}^{+\infty} C_{2i+1}$ is smaller than its limit when $q$
goes to infinity:
$$ \sum_{i=0}^{+\infty} C_{2i+1} \leq \frac{1}{2}\left( e - \frac{1}{e}\right).$$
Now, it is clear that for all $i \geq 0$, $A_i \geq C_0 + C_2 - \sum_{i = 0}^{+\infty} C_{2i +1} = 1 + \frac{q}{2(q+1)}- \frac{1}{2}\left( e - \frac{1}{e}\right)$. Note that this quantity is $\geq 0.15$ for all $q$, and $\geq 0.3$ when $q \geq 23$.

\subsubsection{Probability of finding another factor}\label{proba2}

As usual, we assume that $P$ is a right-divisor of $N \in k^\sigma[X^r]$ irreducible, with $\n(P) = N^e$ and $\deg N = \delta$. We have seen that once we know an irreducible factor of $P$, there is an easy way to factor it without using FirstFactor again.  The following lemma makes this more precise:
\begin{lem}
Let $P = P_2P_1$ with $P_1$ irreducible and $P_2$ reducible, and let $R$ be a random variable following the uniform distribution on $k[X, \sigma]$. Let $A = \text{rgcd}(P, P_1QR)$ and $B = \text{llcm}(A,P_1) = \tilde{B}P_1$. Then the probability that $\tilde{B}$ is an irreducible right factor of $P_2$ is at least $1 - \frac{1}{q^{\delta(e-1)}}$
\end{lem}
\begin{proof}
We work in $k[X,\sigma]/N$. Remark then that $AQ = \text{rgcd}(N,P_1QRQ)$ and that $B = \text{llcm}(AQ,P_1Q)$. We see the multiplication by $RQ$ as an endomorphism $m_{RQ}$ of $k[X,\sigma]Q/N$. Since $R$ follows the uniform distribution, so does $m_{RQ}$. Remark that $m_{RQ}(k[X,\sigma]P_1Q/N)$ is a sub-$\varphi$-module of $k[X,\sigma]/Q$. It is actually equal to $k[X,\sigma]AQ/N$. Indeed, $k[X,\sigma]P_1QRQ \subset k[X,\sigma]AQ$, and $AQ \in k[X,\sigma]P_1QRQ/N$ by definition. Then, we remark that the projection along $k[X,\sigma]P_2$ onto $k[X,\sigma]P_1Q/N$ maps the sub-$\varphi$-module $UQk[X,\sigma]/N$ to $\text{llcm}(U,P_1)Qk[X,\sigma]/N$. In particular, $BQk[X,\sigma]/N$ is the projection of $m_{RQ}(P_1Qk[X,\sigma]/N)$ onto $k[X,\sigma]P_1Q/N$. Therefore, $\tilde{B}$ is an irreducible right-factor of $P_2$ unless $m_{RQ}(k[X,\sigma]P_1Q/N) = k[X,\sigma]P_1Q/N$. Since $m_{RQ}$ is uniformly distributed in the endomorphisms of $D_P$ and $k[X,\sigma]P_1Q/N$ has cardinal $q^{de(e-1)}$ while $D_P$ has cardinal $q^ {de^2}$, this happens with probability $\frac{1}{q^{d(e-1)}}$.
\end{proof}

\subsection{Other algorithms related to factorizations}
In this section, we give some more details on other algorithms that could have interesting applications. The theoretical material on which they rely is only what appears in previous sections.
\subsubsection{Counting factorizations}
This algorithm uses the formula given in corollary \ref{countfac}, and is recursive. First, we give the algorithm computing the number of factorizations of a skew polynomial as a function of its type:\\

\begin{algorithm}[H]
\KwIn{$(\delta, (e_1, \ldots, e_n))$ an integer and a nonincreasing sequence of integers}
\KwOut{The number of factorizations a skew polynomial whose norm is a power of an irreducible of degree $\delta$ and that has dual type $(a_1, \ldots, a_n)$}
$h=1$\;
\For{$1 \leq i \leq n$}{
  \If{$(i = n)$ or $(a_i > a_{i+1})$}{
  $j = \min \{l~|~a_l = a_i\}$\;
   $h=h \times \text{CountFactorizationStep}(\delta, (a_1, \ldots, a_{i-1}, a_i -1, a_{i+1}, \ldots, a_n))$\\$\hphantom{h=h}\times (q^{\delta l} + \cdots + q^{\delta i})$\;
  }
}
\Return{$h$\;}
\caption{CountFactorizationsStep}
\end{algorithm}
The following algorithm gives the number of factorizations of a given skew polynomial as the product of its leading coefficient and monic irreducible skew polynomials. We assume that we have a function \verb"DualType" that computes the dual of a nondecreasing sequence of integers.\\

\begin{algorithm}[H]
\KwIn{$P \in k[X,\sigma]$}
\KwOut{The number of factorizations of $P$ as a product of its leading coefficient and monic irreducible polynomials}
$N = \n(P)$\;
$([N_1,d_1], \cdots, [N_t, d_t]) = \text{Factorization}(N)$\;
$h=1$; \, $\tau = 0$\;
\For{$1 \leq i \leq t$}{
  $\delta=\deg(N_i)$\;
  $A = \text{rgcd}(P, N_i)$;\, $j=1$\;
  \While{$A \neq 1$}{
    $e_j = \frac{\deg A}{\delta}$\;
    $P= P/A$\;
    $A=\text{rgcd}(P,N_i)$\;
    $j= j+1$\;
    }
  $\tau = \tau + e_1 + \cdots + e_{j-1}$\;
  $a = \text{DualType}(e_1, \ldots, e_{j-1})$\;
  $h= \frac{1}{(e_1 + \cdots + e_{j-1}) !} \cdot \text{CountFactorizationsStep}(\delta, a)$\;
}
$h = \tau! \times h$\;
\Return{$h$}\;
\caption{CountFactorizations}
\end{algorithm}

\subsubsection{Random factorizations}
For some applications of skew polynomials, it could be interesting to have an algorithm that returns a factorization of a given skew polynomial $P$, following the uniform distribution on all factorizations of this skew polynomial. In this section, we describe such an algorithm.

Since we do not want to simply list all factorizations of $P$ and pick one randomly (because there can be so many factorizations, so this would have a very bad complexity), we want an algorithm that can simulate the uniform distribution on the right-factors of $P$. Let us first explain how to do this when $P$ has type $(e)$.

Assume $P$ has type $(e)$ and norm $N^e$, and let $E = k^\sigma[X^r]/(N)$. As usual, $d = \deg N$ and $q = \#k^\sigma$. Suppose that we know one irreducible right-factor $P_0$ of $P$. This factor correponds to a $E$-line in $D_P = k[X,\sigma]/k[X,\sigma]P$. The orbit of this line under the action of $\End_\varphi(D_P)$ is the set of all $E$-lines in $D_P$, corresponding to all irreducible right-divisors of $P$. We now want to find one element $u \in \End_\varphi(D_P)$ such that $\End_\varphi(D_P) = E(u)$, so that in order to simulate the uniform distribution on the irreducible right-divisors of $P$, it is enough to simulate the uniform distribution on the polynomials of degree $<e$ with coefficients in $E$ and compute the image of our line under the action of $M(u)$, where $M$ is the polynomial that we get. Let us estimate the probability to find such a $u$. Since $\End_\varphi(D_P) \simeq \mathcal{M}_e(E)$, we want to know the probability that, for a fixed nonzero vector $x \in E^e$, an element $u \in \mathcal{M}_e(E)$ admits an element of the line $Ex$ as a cyclic vector. The number of $u \in \mathcal{M}(E)$ that have this property is $(q^{d}-1)(q^{de} - q^d)\cdots (q^{de} -q^{d(e-1)})$. Hence the probability to find a $u$ with the desired property is
$$ \left( 1 - \frac{1}{q^d}\right)^2 \left(1 - \frac{1}{q^{2d}}\right) \cdots \left(1 - \frac{1}{q^{de}} \right),$$
which is greater than $0.53\left( 1 - \frac{1}{q^d} \right)$ by \cite{pra}, Lemma 2.2 (applied to computing the value of the infinite product when $q^d = 4$).

Once we know how to randomly get a divisor of a polynomial of type $(e)$, the idea of the algorithm to get a random factorization of any skew polynomial $P$ is the following: compute the type of $P$, and randomly choose an irreducible divisor $N_i$ of $\n(P)$ (with uniform distribution). Compute $Q = \text{rgcd}(N_i,P)$ and randomly find a right-irreducible divisor $P_1$ of $Q$ with the previous algorithm. Let $Q_1P_1 = Q$. Compute the type of $Q_1$. Keep the factor $P_1$ with the same probability as the ratio of irreducible right-divisors of $Q$ yielding a left-divisor of the type of $Q_1$ (this can be done counting the factorizations of $Q$ and $Q_1$, which depends only on their types). Write $P = RP_1$ and randomly factor $R$. By construction, it is clear that the factorization we get in the end is uniformly distributed among all factorizations of $P$.

\subsection{Implementation}

All algorithms presented in this article were implemented is {\sc sage}.
The source code is available on the CETHop website at the URL:

\begin{center}
\url{http://cethop.math.cnrs.fr/documents/skew_polynomials-sage.tgz}
\end{center}

\noindent
A {\sc magma} package is also available; it includes some of the 
algorithms presented here (and, in particular, the factorization 
algorithm). It can be downloaded at the URL:

\begin{center}
\url{http://cethop.math.cnrs.fr/documents/skew_polynomials.m}
\end{center}

%\noindent
%Thanks to the Sage notebook, online demonstrations (including timings)  
%are also available on the CETHop website:
%
%\begin{center}
%\url{http://cethop.math.cnrs.fr/enligne/}
%\end{center}

\bibliographystyle{amsalpha}
\bibliography{fastfactor}
\end{document}